\documentclass[reqno, 11pt]{amsart}
\usepackage{graphicx, enumerate, enumitem, amsmath, amssymb,amsthm,hyperref,amscd,xcolor}
\usepackage{mathabx,manfnt,braket,mathrsfs,comment}
\usepackage{ulem}
\usepackage[makeroom]{cancel}
\usepackage[utf8]{inputenc}

\newtheorem{thm}{Theorem}[section]
\newtheorem*{thm*}{Theorem}
\newtheorem{prop}{Proposition}[section]
\newtheorem{lem}[prop]{Lemma}
\newtheorem{cor}[prop]{Corollary}
\theoremstyle{definition}
\newtheorem{defn}[thm]{Definition}

\theoremstyle{remark}
\newtheorem{rem}[thm]{Remark}

\numberwithin{equation}{section}

\newcommand{\norm}[1]{\left\Vert#1\right\Vert}
\newcommand{\abs}[1]{\left\vert#1\right\vert}
\newcommand{\cx}{{\mathbb{C}}}

\newcommand{\rl}{{\mathbb{R}}}
\newcommand{\tensor}{\otimes}

\newcommand*{\Cdot}{\raisebox{-0.25ex}{\scalebox{1.65}{$\cdot$}}}

\newcommand{\tmop}[1]{\ensuremath{\operatorname{#1}}}
\renewcommand{\Re}{\tmop{Re}}

\newcommand{\ol}{\overline}

\newcommand{\dbar}{\ol\partial}
\newcommand{\wt}{\widetilde}

\DeclareMathOperator{\dist}{dist}

\newcommand\gipr[1] {\left( #1 \right)}

\newcommand{\eps}{\varepsilon}
\newcommand{\ddbar}{\partial\bar\partial}
\DeclareMathOperator{\dom}{dom}
\DeclareMathOperator{\range}{range}

\title[A Modified Morrey-Kohn-H\"ormander Identity and Applications]{A Modified Morrey-Kohn-H\"ormander Identity and Applications to the $\dbar$-Problem}

\author{Debraj Chakrabarti}
\address{Department of Mathematics, Central Michigan University, Mt. Pleasant,  MI 48859,  USA}
\email{chakr2d@cmich.edu}

\author{Phillip S. Harrington}
\address{SCEN 309, 1 University of Arkansas, Fayetteville, AR 72701, USA}
\email{psharrin@uark.edu}

\thanks{Debraj Chakrabarti was partially supported by NSF grant  no.~1600371.}
\begin{document}

\begin{abstract}
We prove a  modified form of the classical Morrey-Kohn-H\"ormander identity, adapted to  pseudoconcave boundaries. Applying this result
to an annulus between two bounded pseudoconvex domains in $\cx^n$, where the inner domain has
 $\mathcal{C}^{1,1}$ boundary, we show that the $L^2$ Dolbeault cohomology group
in bidegree $(p,q)$ vanishes if  $1\leq q\leq n-2$  and is Hausdorff and infinite-dimensional
if $q=n-1$, so that the Cauchy-Riemann operator has closed range in each bidegree.
As a dual result, we prove that the Cauchy-Riemann operator is solvable in the $L^2$ Sobolev space $W^1$ on any pseudoconvex domain with $\mathcal{C}^{1,1}$ boundary.  We also generalize our results to annuli between domains which are weakly $q$-convex in the sense of Ho for appropriate values of $q$.
\end{abstract}

\maketitle

\section{Introduction}
In the theory of holomorphic functions of several variables,  pseudoconvex domains play a central
role, since they are precisely the domains of holomorphy: on a pseudoconvex domain (open connected set)  $\Omega\subset\cx^n$, there is a holomorphic function which cannot be extended past each boundary point of $\Omega$, even locally.  One of several possible definitions of pseudoconvexity of $\Omega$  is that $\Omega$ admits a strictly plurisubharmonic exhaustion function. From the early days of the theory, it has also been clear that some of the nice
properties of pseudoconvex domains extend to non-pseudoconvex domains satisfying weaker
convexity conditions (cf. \cite{AnGr62}). The classical approach to the existence  and finiteness
theorems through patching of local results using sheaves was soon supplemented by powerful
methods based on $L^2$-estimates on the $\dbar$- and $\dbar$-Neumann problems (\cite{Koh63, Hor65, andreottivesentini}). H\"ormander in   \cite{Hor65} showed that the Cauchy-Riemann operator is solvable in $L^2_{p,q}(\Omega)$ for all $0\leq p\leq n$ and $1\leq q\leq n$ provided $\Omega$ is  pseudoconvex and bounded.  Significantly, H\"ormander's results require no boundary smoothness for $\Omega$. Estimates and existence results  were
also obtained in non-pseudoconvex domains satisfying appropriate convexity conditions on
the Levi form of the boundary or on exhaustion functions, such as the conditions variously called $a_q$, $A_q$, or $Z(q)$ (see \cite{FoKo72,Hor65}),  {weak $Z(q)$ (see \cite{HaRa15}),} or weak $q$-convexity (see  \cite{Ho91}).  Given the centrality of $L^2$-methods in modern complex analysis,
it is of great interest to try to extend them to a wider classes of domains, with the fewest regularity assumptions on the boundary and the widest range of admissible convexity conditions on the domain  possible.

An interesting class of non-pseudoconvex domains is that of  annuli.
By an {\em annulus} we mean a domain $\Omega$ in $\cx^n$ of the form $\Omega=\Omega_1\setminus \ol{\Omega_2}$, where $\Omega_1$ is a bounded domain in $\cx^n$ and $ \Omega_2$ is an open subset (not necessarily connected)
in $\cx^n$ such that $\Omega_2\Subset \Omega_1$, i.e., $\Omega_2$ is relatively compact
in $\Omega_1$.  We will refer to $\Omega_1$ as the {\em envelope}
of the annulus $\Omega$ and $\Omega_2$ as the {\em hole}.  When $n\geq 3$, and both the envelope and hole are strongly pseudoconvex with smooth boundary, the annulus satisfies condition $Z(q)$ for $1\leq q \leq n-2$ and therefore we have subelliptic estimates for the $\dbar$-Neumann problem \cite{FoKo72}. For  pseudoconvex envelope and hole with smooth boundaries, the problem was studied by Shaw in
\cite{Sha85}, where she showed that the $\dbar$-operator has closed range in $L^2_{p,q}(\Omega)$ for all $0\leq p\leq n$ and $1\leq q\leq n-2$.  In \cite{Hor04}, H\"ormander
examined carefully the special case where the hole and the envelope are concentric balls, and
showed that for $1\leq q\leq n-2$ the range of the $\dbar$-operator coincides with the $\dbar$-closed forms in $L^2_{p,q}(\Omega)$, but for $q=n-1$, there is an infinite dimensional  cohomology group, which forms the obstruction to the solvability of the $\dbar$-problem. In
\cite{Sha10}, Shaw showed that the same statements hold in general annuli with smoothly bounded pseudoconvex hole and envelope.
 In \cite{CLS18}, the first author, together with Laurent-Thi\'ebaut and Shaw, proved that the Cauchy-Riemann operator has closed range in $L^2_{0,1}(\Omega)$  on a class of annuli $\Omega$ in which the hole has only Lipschitz regularity.

H\"ormander's argument on pseudoconvex domains requires no boundary smoothness because pseudoconvex domains can be exhausted by smoothly bounded pseudoconvex domains,
and the closed range estimates on each domain in the exhaustion have uniform constants.  In contrast, Shaw's proof seems to require a boundary of at least class $\mathcal{C}^3$, since a key computation (\cite[ equation (3.18)] {Sha85}) involves commuting the adjoint of a tangential derivative with another derivative.  In \cite{HaRa17}, the second author and Raich consider a more general class of non-pseudoconvex domains and find that in such cases a $\mathcal{C}^4$ boundary may be required.

In this paper, we study function theory on annuli using an alternative approach different from those in \cite{Sha85, Hor04, CLS18} but somewhat related to the one in \cite{LiSh13}. Our main tool is a connection between $W^1$ estimates
on the hole and $L^2$-estimates on the annulus, which is  reminiscent of
{\em Alexander duality}
in topology.  Alexander duality relates the \v{C}ech cohomology of a compact subset $A$ of the sphere  or Euclidean space with the homology groups (or cohomology groups, via Poincaré duality) of  the complement of $A$ (see, e.g., \cite[pp. 352--353]{bredon}). In order to use this technique, we
need to establish existence and estimates  for the $\dbar$-operator on an annulus,
where we must use {\em mixed} boundary conditions for the $\dbar$-operator. While this problem was studied in \cite{LiSh13} assuming $W^1$ estimates in the hole, we first establish the required estimates directly on an annulus with $\mathcal{C}^2$-smooth hole, and use the estimate on the annulus to obtain the $W^1$-estimates on the hole,  at which point an exhaustion argument allows us to deal with $\mathcal{C}^{1,1}$-smooth holes. This brings us to the main innovation
of this paper: we carry out a careful computation of the error terms that arise when computing
 closed range estimates with a canonical defining function (the signed distance function) to show that the third derivatives of the defining function which arise will
 ultimately cancel each other.  In our basic identity for the inner boundary of the annulus (Theorem \ref{thm:pseudoconcave_identity}), we find that the only significant error term is a
  curvature term  $\eta$ (defined in \eqref{eq-eta_squared})  depending entirely on second derivatives of the signed distance function. As a consequence, we are not only able to prove solvability of the $\dbar$-problem
  for annuli with holes that are only $\mathcal{C}^{1,1}$-smooth, but in principle estimate the
  constant in the $L^2$-estimate on the annulus, a problem which most likely cannot be
  handled by the method of \cite{Sha85}.

It should be noted that special cases of the connection between $W^1$-estimates on the hole and
$L^2$-estimates on the annulus have been studied before. In these cases, the relation between the
 $L^2$-estimates and $W^1$-estimates  is established in particular degrees using Hartogs
 phenomenon (see, e.g., \cite[Proposition~4.7]{LaSh13}).  In \cite{CLS18}, such a relation was
used to obtain $W^1$-estimates for the $\dbar$-problem on the polydisc in $\cx^n$ in degree $(0,n-1)$ . However, this technique
based on Hartogs phenomenon cannot be used to obtain estimates in  any other degree.

Most of the results of this paper have generalizations to domains in Stein manifolds.  {Given a Stein manifold $M$, we have a strictly plurisubharmonic exhaustion function $\varphi$ which can be used to generate a K\"ahler form $\omega=i\ddbar\varphi$.  The methods of proof in this paper can be adapted to this case using the given weight function $\varphi$ in place of $|z|^2$.  See \cite{HaRa15} for details.  As demonstrated in \cite{HaRa15}, this extra flexibility can be helpful when dealing with hypotheses such as ours that depend on the choice of metric, since even in $\mathbb{C}^n$ there may be examples in which our hypotheses fail under the Euclidean metric but are satisfied for a different K\"ahler metric.}
We choose to
present them in the setting of $\cx^n$ for simplicity.

\subsection{Acknowledgements} We thank Peter Ebenfelt and László Lempert for their valuable comments.
\section{Preliminaries and Statement of Results}
\label{sec:preliminaries}
\subsection{Weakly \texorpdfstring{$q$}{q}-convex Domains} Our results will be stated
in terms of {\em  weakly $q$-convex domains}, a class of domains introduced by Ho in  \cite{Ho91}, where he observed that  many of H\"ormander's arguments in \cite{Hor65} and Shaw's arguments in \cite{Sha85} would follow for domains which were not necessarily pseudoconvex.  For a domain $D\Subset\cx^n$, and for a fixed $1\leq q\leq n-1$, Ho defined $D$ to be  {\em weakly $q$-convex} if  at any boundary point the sum of any $q$ eigenvalues of the Levi-form is positive (technically this is an equivalent formulation proven in \cite[Lemma 2.2]{Ho91}).  We will adapt Ho's definition to non-smooth domains, as follows. Recall that the  {\em signed distance function} $\rho$   of  an open set  ${D}\subset\cx^n$ is the function
\begin{equation}\label{eq-rhodef}
\rho(z)= \begin{cases}\phantom{-} \dist(z,b{D}) &\text{outside  } {D} \\
-\dist(z,b{D})& \text{inside  }{D}.\end{cases}
\end{equation}

\begin{defn}
\label{defn:weak_q_convex}
  For $D\subset\mathbb{C}^n$ open and $1\leq q\leq n$, we say that a continuous {function} $\varphi:D\rightarrow\mathbb{R}$ is \emph{ $q$-subharmonic} on $D$ if  $\varphi$ satisfies the sub-mean value property on any $q$-dimensional polydisc in $D$.  By the last statement we mean the following: if $P_q(0,r)=B(0,r)\times\dots\times B(0,r)\subset\mathbb{C}^q$ is a $q$-dimensional polydisc and $\psi:\ol{P_q(0,r)}\rightarrow D$ is a holomorphic isometry, i.e., a mapping of the form $\psi(z)= Uz+b$ where $U$ is  an  $n\times q$ matrix such that $U^*U$ is the identity on $\cx^q$,  and $b\in\cx^n$, then
  \[ \varphi(\psi(0)) \leq \frac{1}{(2\pi)^q}\int_0^{2\pi}\dots \int_0^{2\pi} \varphi(\psi(re^{i\theta_1}, \dots, r e^{i\theta_q}))d\theta_1\dots d\theta_q.\]
  For an open {set} ${D}\subset\mathbb{C}^n$ and $1\leq q\leq n-1$, we say that ${D}$ is weakly $q$-convex if there exists a constant $C>0$ and a neighborhood $U$ of $b{D}$ such that $-\log(-\rho)+C|z|^2$ is $q$-subharmonic on $U\cap{D}$, where $\rho$ is the signed distance function for ${D}$ (see \eqref{eq-rhodef} above for its definition).
  For $q=0$ or $q=n$, we adopt the convention that no {open set} is weakly $0$-convex and all {open sets} are weakly $n$-convex.
\end{defn}

 Several remarks are in order regarding these definitions. As in \cite{Hor90}, it suffices to consider upper semicontinuous functions $\varphi$ in the definition of $q$-subharmonicity, but we will not need this generality in the following.
When $\varphi$ is $\mathcal{C}^2$, $\varphi$ is $q$-subharmonic if and only if the sum of any $q$ eigenvalues of the complex Hessian of $\varphi$ are positive at every point in ${D}$.  This follows from  \cite[Theorem 1.4]{Ho91}. Notice consequently that a 1-subharmonic function is just a (continuous) plurisubharmonic function, and a weakly 1-convex domain is just a
pseudoconvex domain.

 On $\mathcal{C}^3$ domains, our definition of  being
 weakly $q$-convex via exhaustion agrees with Ho's definition via \cite[Theorem 2.4]{Ho91} (see also the remark following the proof).
 In fact, Ho shows that on $\mathcal{C}^3$ domains, one can use any defining function $\rho$.  We choose to use the signed distance function for two reasons.  In the first place, this is the natural analog of the pseudoconvex case, in which Oka's Lemma tells us that $-\log(-\rho)$ is subharmonic on any pseudoconvex domain.  In the second place, on $\mathcal{C}^{1,1}$ domains the signed distance function is known to be $\mathcal{C}^{1,1}$ on a neighborhood of the boundary (this is also shown by Krantz and Parks \cite{KrPa81}, since Federer has already established that $\mathcal{C}^{1,1}$ domains have positive reach \cite{Fed59}).

 Notice also that the notions of $q$-subharmonicity and weak $q$-convexity, as defined, are not  biholomorphically invariant except when $q=1$. They are of course invariant under the holomorphic isometries, as well as the slightly larger group of transformations of $\cx^n$ of  the form
 $\psi(z)=c Uz+b$ where $c\not=0$, $U$ is an $n\times n$ unitary matrix and $b\in \cx^n$.

\subsection{Results on the \texorpdfstring{$\dbar$}{dbar}-problem}

Recall that a domain $D$ is said to have $\mathcal{C}^{1,1}$ boundary if there exists a defining function $\rho$ for $D$ that is $\mathcal{C}^1$ with Lipschitz first derivatives in a neighborhood of the boundary.  Krantz and Parks \cite{KrPa81} show that we can, in fact, take $\rho$ to be the signed distance function for $D$.

We will use $W^1_{p,q}(D)$ to denote the space of $(p,q)$-forms with coefficients in the $L^2$-Sobolev space $W^1(D)$.

Our first result is the following:
\begin{thm}
\label{thm:W^1_solvability}
  Let $D$ be a bounded domain in $\mathbb{C}^n$, $n\geq 2$, such that  $1\leq q\leq n$ and $D$ is a weakly $q$-convex domain with $\mathcal{C}^{1,1}$ boundary.  Then, for
  $0\leq p\leq n$, there exists a constant $C>0$  such that for every $f\in W^1_{p,q}(D)$ satisfying $\dbar f=0$, there exists $u\in W^1_{p,q-1}(D)$ satisfying $\dbar u=f$ and
  \begin{equation}
  \label{eq:W^1_estimate}
    \norm{u}_{W^1}\leq C\norm{f}_{W^1}.
  \end{equation}
\end{thm}
In the case when $q=1$ (a pseudoconvex domain with $\mathcal{C}^{1,1}$ boundary), this is contained in the results obtained by the second author in \cite{Har09b}. In \cite{Har09b}, the solution was a weighted canonical solution, but in our result, the solution
is obtained in a different way. We start from a solution of the $\dbar$-problem with mixed boundary conditions in an annulus with $D$ as its hole. As a consequence, we obtain an explicit upper bound \eqref{eq:W^1_constant_estimate}	for the constant $C$ in terms of the geometry of the domain $D$.

We now go on to results on annuli.  We have the following closed-range result:

\begin{thm} \label{thm:q_n-q} Let $\Omega$ be an annulus in $\cx^n$, with $n\geq 3$ such that for some $q$, $1\leq q \leq n-1$,  its envelope is weakly $q$-convex, and the hole  is weakly $(n-q)$-convex and has $\mathcal{C}^{1,1}$ boundary. Then for $0\leq p \leq n$, the $\dbar$-operator has closed range in $L^2_{p,q}(\Omega)$.
\end{thm}

Thanks to the closed-range property, the $L^2$-Dolbeault cohomology group
\[ H^{p,q}_{L^2}(\Omega)= \frac{\ker \dbar: L^2_{p,q}(\Omega)\to L^2_{p,q+1}(\Omega)}{{\rm range}\, \dbar: L^2_{p,q-1}(\Omega)\to L^2_{p,q}(\Omega) }\]
is Hausdorff in its natural topology (in fact,  for any domain $D$, the space $H^{p,q}_{L^2}(D)$ is Hausdorff if and only if $\dbar$ has closed range in $L^2_{p,q}(D)$). It does not necessarily follow that $H^{p,q}_{L^2}(\Omega)$ is trivial, and we will in fact see a case where we have closed range in $L^2_{p,n-1}(\Omega)$, but $H^{p,n-1}_{L^2}(\Omega)$ is infinite dimensional.  It is of course of great interest to compute this cohomology, which depends on  geometric properties more refined
than weak $q$-convexity.  However, we do have the following simple condition for the vanishing of this cohomology:

\begin{thm}\label{thm:q-convex}  Let   $\Omega$ be an annulus  in $\mathbb{C}^n$, $n\geq 3$,  such that  for some $1 \leq q \leq n-2$, the envelope is weakly $q$-convex and the
hole has $\mathcal{C}^{1,1}$ boundary and is weakly $(n-q-1)$-convex. Then for $0\leq p \leq n$, we have
\[ H^{p,q}_{L^2}(\Omega)=0.\]
\end{thm}

On the other hand, we have the following situation in which the cohomology is infinite dimensional:
\begin{thm}\label{thm-infinitedim}
Let $\Omega$ be an annulus in $\mathbb{C}^n$, $n\geq 3$, such that the envelope is weakly $(n-1)$-convex and the hole has $\mathcal{C}^{1,1}$ boundary and is pseudoconvex.  Then for $0\leq p\leq n$, $H_{L^2}^{p,n-1}(\Omega)$ is Hausdorff and infinite-dimensional.
\end{thm}

In fact, we will show that there is a natural conjugate linear isomorphism between $H^{p,n-1}_{L^2}(\Omega)$ with $W^1_{n-p,0}(\Omega_2)\cap\ker\dbar$, the space of holomorphic $(n-p,0)$-forms with coefficients in the Sobolev space $W^1(\Omega_2)$.

Unlike the $q=n-1$ case, the condition of weak $(n-q-1)$-convexity on the inner boundary in Theorem \ref{thm:q-convex} is certainly not sharp, since weak $q$-convexity is not a biholomorphic invariant while vanishing $L^2$-cohomology is invariant (assuming the biholomorphism extends smoothly to the boundary).  However, we do have the following immediate consequence of our Theorem \ref{thm:q_n-q} and Theorem 3.1 in \cite{Hor04}:
\begin{cor}\label{cor-hormander}
Let $\Omega$ be an annulus in $\mathbb{C}^n$, $n\geq 3$, such that for some $q$, $1\leq q\leq n-1$, the envelope is weakly $q$-convex and the hole is weakly $(n-q)$-convex with $\mathcal{C}^{1,1}$ boundary.  Suppose further that there exists a point  in $b\Omega_2$ such that $b\Omega_2$ is $\mathcal{C}^3$ in a neighborhood of this point, and  at which the Levi-form for the hole has exactly $q$ positive eigenvalues and $n-q-1$ negative eigenvalues.  Then for $0\leq p\leq n$, $H_{L^2}^{p,q}(\Omega)$ is Hausdorff and infinite-dimensional.
\end{cor}

When the Levi-form of the hole has $n-q-1$ negative eigenvalues at a point, the hole can not be weakly $(n-q-1)$-convex.  Hence, this corollary is consistent with our Theorem \ref{thm:q-convex}.  Our Theorem \ref{thm-infinitedim} generalizes the $q=n-1$ case of this corollary to the case in which the hole has only $\mathcal{C}^{1,1}$ boundary.

\begin{proof}[Proof of Corollary~\ref{cor-hormander}]
Denote the envelope of $\Omega$ by $\Omega_1$ and the hole by $\Omega_2$.   Now \cite[Theorem 3.1]{Hor04} states that if there exists a point on $b\Omega_2\subset b\Omega$ at which the Levi-form for $\Omega$ has exactly $q$ negative and $n-q-1$ positive eigenvalues and $b\Omega$ is $\mathcal{C}^3$ in a neighborhood of this point, then $H_{L^2}^{p,q}(\Omega)$ is infinite dimensional whenever $\dbar$ has closed range in $L^2_{p,q}(\Omega)$.  Since the Levi-form for $\Omega$ has $q$ negative and $n-q-1$ positive eigenvalues on $b\Omega_2$ precisely when the Levi-form for $\Omega_2$ has $q$ positive and $n-q-1$ negative eigenvalues, this requirement is guaranteed by our hypotheses.  By Theorem \ref{thm:q_n-q}, we know that that $\dbar$ has closed range in $L^2_{p,q}(\Omega)$, so the conclusion follows.
\end{proof}


\subsection{Statement of Main Identity}
\label{sec:stmtmainidentity}
The key tool in proving the results from the previous section will be an adaptation of the classical Morrey-Kohn-H\"ormander Identity for pseudoconcave boundaries.  Such an adaptation
 has already been computed in the form of an estimate in \cite{Sha85},
 for example.  Our primary contribution will be to carry out a careful computation
 of the error terms in this estimate to show that they are uniformly bounded
 on $\mathcal{C}^2$ domains, which will eventually allow us to obtain estimates on
 $\mathcal{C}^{1,1}$ domains.  Before stating our key identity, we fix some notation that we will use throughout this paper.

Let ${D}$ be a bounded domain in $\cx^n$ with $\mathcal{C}^2$ boundary.
By work of Krantz and Parks \cite{KrPa81}, we know that the signed distance function $\rho$  of $D$ (see \eqref{eq-rhodef}) is $\mathcal{C}^2$ on a neighborhood of $b{D}$.  Let $U$ be a neighborhood of $b{D}$ on which $\rho$ is $\mathcal{C}^2$. After shrinking $U$, we
will assume that $\rho $ is in fact $\mathcal{C}^2$ on  the closure $\ol{U}$.

On $U$,  \cite[Theorem 4.8 (3)]{Fed59} immediately gives us
\begin{equation}\label{eq-drho}
\abs{d\rho}\equiv 1.
\end{equation}
  Let
\begin{equation}\label{eq-Ndef}
N=4 \sum_{j=1}^n \rho_{\bar{j}}\frac{\partial}{\partial z_j}
\end{equation}
be a $(1,0)$ vector field on $U$ which is  normal   to the boundary $b{D}$.
For $j=1,\dots, n$, introduce the tangential vector fields $L_j$ on $U$ by setting
\begin{equation}\label{eq-Ljdef}
L_j = \frac{\partial}{\partial z_j} - \rho_j N.
\end{equation}
On $U$ we may define the functions:
\begin{equation} \label{eq-tracelevi}
\tau=\sum_{j=1}^n\ol{ L_j}\rho_j
\end{equation}
and
\begin{equation} \label{eq-eta_squared}
\eta^2 = \sum_{j,k=1}^n\abs{L_j\rho_k}^2 .
\end{equation}
We will see that $\tau|_{bD}$ represents the trace of the Levi-form (in \eqref{eq:trace_levi_form}, for example).  The quantity $\eta^2$ will play a key role in Theorem \ref{thm:pseudoconcave_identity}.  It depends on those components of the second fundamental form which are not biholomorphically invariant, so it is a Hermitian invariant, but it is not a local CR-invariant (at any given $p\in bD$, one can choose local holomorphic coordinates in which $\eta^2$ vanishes at that $p$).  Hence, the pointwise value of $\eta^2$ seems to have less intrinsic meaning than the quantity $\sup_{U\backslash D}\eta^2$ (see \eqref{eq-Bdef} below).  When $D$ is convex, one can check that $\sup_{U\backslash D}\eta^2=\sup_{bD}\eta^2$, but this does not seem to be true in general, so the size of the neighborhood $U$ is also relevant in our estimates (although this can still be understood geometrically using the concept of reach as in \cite{Fed59}).  When $D$ is the ball, one can check that $\eta^2\equiv 0$.

For any weight function $\phi\in\mathcal{C}^2(U)$, we also define
\[
  \gamma_\phi = \sum_{j=1}^n \ol{L_j}\phi_j.
\]
For a differential operator $X$, let $X^*$ denote the adjoint of $X$ in the weighted space $L^2({D}, e^{-\phi})$.

In general, we adopt the notation of Straube \cite{Str10} and Chen-Shaw \cite{ChSh01} when studying the $L^2$ theory for $\dbar$.  If $u$ is a $(p,q)$-form, we will write $u = \sum_{\abs{J}=p,\abs{K}=q}' u_{J,K}dz^J\wedge d\bar{z}^K$ where $J$ is a multi-index of length $p$, $K$ is a multi-index of length $q$, and $\sum'$ indicates that we are summing only over increasing multi-indices.  We extend $u_{J,K}$ to nonincreasing multi-indices by requiring $u_{J,K}$ to be skew-symmetric with respect to its indices.  We let a scalar linear differential operator $X$ act on a differential form $u = \sum_{\abs{J}=p,\abs{K}=q}' u_{J,K}dz^J\wedge d\bar{z}^K$ coefficientwise, i.e., $Xu= \sum_{\abs{J}=p,\abs{K}=q}' (X u_{J,K})dz^J\wedge d\bar{z}^K$.

We denote by $H\psi$ the complex Hessian of the function $\psi$ on $\cx^n$, i.e., the Hermitian matrix $\psi_{j\ol{k}}= \frac{\partial^2 \psi}{\partial z_j\partial \ol{z_k}}$. Then given a $(p,q)$-form $u$, we have the natural
action
\begin{equation}\label{eq-Hessian}
 H\psi(u,u) = {\sideset{}{'}\sum_{\abs{J}=p,\abs{K}=q-1}} \sum_{j,k=1}^n \psi_{j\ol{k}} u_{J,jK} \ol{u_{J,kK}}
\end{equation}
making $H\psi$ into a Hermitian form acting on the space of $(p,q)$-forms.

Given a weight function $\phi\in\mathcal{C}^2(\overline{D})$, we have the standard $L^2$ inner product on functions
\[
  (u,v)_{L^2({D},e^{-\phi})}=\int_{D} u\bar v e^{-\phi}dV
\]
with norm $\norm{u}_{L^2({D},e^{-\phi})}^2=(u,u)_{L^2({D},e^{-\phi})}$.  This easily extends to $(p,q)$-forms via
\[
  (u,v)_{L^2({D},e^{-\phi})}={\sideset{}{'}\sum_{\abs{J}=p,\abs{K}=q}}(u_{J,K},v_{J,K})_{L^2({D},e^{-\phi})}
\]
with the associated norm.  When the domain ${D}$ is clear from context, we will often abbreviate
\[
  (u,v)_\phi=(u,v)_{L^2({D},e^{-\phi})}\text{ and }\norm{u}_\phi=\norm{u}_{L^2({D},e^{-\phi})}.
\]

With this notation in place, we are ready to state our main identity:
\begin{thm}\label{thm:pseudoconcave_identity}
  Let ${D}\subset\mathbb{C}^n$ be a bounded domain with $\mathcal{C}^2$ boundary, let $\rho$ be the defining function for $D$ given by \eqref{eq-rhodef}, and let $U$ be a bounded neighborhood of $\ol{D}$ such that $\rho$ is $\mathcal{C}^2$ on $\ol U\backslash D$.
   Let $\phi\in \mathcal{C}^2(U{\setminus}{D})$ be a real valued function.  For $1\leq q\leq n$ and $0\leq p\leq n$, let the $(p,q)$-form $u\in \mathcal{C}^1_{p,q}(\ol{U}{\setminus}{D})\cap\dom\dbar^*_\phi$ have compact support in  $U{\setminus}{D}$.  We have
  \begin{align}
  \norm{\dbar u}_{\phi}^2+ \norm{\dbar^*_\phi u}^2_{\phi} &=  -\int_{b{D}}(H\rho-\tau I)(u,u) e^{-\phi}d\sigma\nonumber\\
  &+\int_{U{\setminus}\ol{D}}(H\phi-\gamma_\phi I)(u,u) e^{-\phi} dV-4 \norm{\eta u}_{\phi}^2\nonumber\\
  &+\frac{1}{4}\norm{(\ol{N}+4\tau)u}^2_{\phi}+\sum_{j=1}^n \norm{(\ol{L_j}^*-4\rho_j\tau) u}_{\phi}^2.   \label{eq:pseudoconcave_identity}
  \end{align}
\end{thm}

  For the sake of clarity, we emphasize that $U{\setminus}{D}$ is {\em not} an open set,
   so a form with compact support in $U{\setminus}{D}$ does not necessarily vanish on $b{D}$.


\section{Integration by Parts and the Gradient}
\label{sec:int_by_parts}

\label{sec:background}
We continue to use the notation established in Section~\ref{sec:stmtmainidentity}.  In this section, we will prove the following key identity, which will then be used in the next section in the proof of Theorem~\ref{thm:pseudoconcave_identity}:
\begin{lem}
\label{lem:int_by_parts_gradient}
  Let ${D}\subset\mathbb{C}^n$ be a bounded domain with $\mathcal{C}^2$ boundary, let $\rho$ be the signed distance function for ${D}$, and let $U$ be a neighborhood of $\ol{D}$ such that $\rho$ is $\mathcal{C}^2$ on $U\backslash\overline D$.  Let $\phi\in \mathcal{C}^2(\ol{U}{\setminus}{D})$ be a real valued function.  For any $v\in \mathcal{C}^1_0(U)$, we have
  \begin{multline}
  \label{eq:int_by_parts_gradient}
    \sum_{j=1}^n \norm{\ol{L_j}v}_{\phi}^2  =\sum_{j=1}^n \norm{(\ol{L_j}^*-4\rho_j\tau) v}_{\phi}^2+  4\norm{\tau v}_{\phi}^2\\ + \int_{b\Omega} \tau\abs{v}^2 e^{-\phi} d\sigma  + 2 \Re \gipr{\tau v, \ol{N} v}_{\phi} -\gipr{\left(\gamma_\phi + 4 \eta^2\right)v,v}_{\phi}.
  \end{multline}
\end{lem}

\subsection{Reduction to \texorpdfstring{$\mathcal{C}^3$}{C3} Domains}
\label{sec:reduction_to_C3}

Suppose that Lemma~\ref{lem:int_by_parts_gradient} has already been proven for domains with
 $\mathcal{C}^3$ boundaries.  Let ${D}$, $U$, and $\rho$ be as in the
  statement of Lemma~\ref{lem:int_by_parts_gradient}.  Let $\chi\in\mathcal{C}^\infty_0(\mathbb{C}^n)$ be a nonnegative function satisfying $\int\chi=1$, and for $j\in\mathbb{N}$ define $\chi_j(z)=j^{2n}\chi(j z)$ so that $\int\chi_j=1$ for every $j\in\mathbb{N}$.  We may assume that $\chi$ is supported in a ball of radius $1$, so $\chi_j$ is supported in a ball of radius $j^{-1}$.  Set
\[
  U_j=\set{z\in U:\dist(z,bU)>j^{-1}}.
\]
On $U_j$, let $\tilde\rho_j=\rho\ast\chi_j$.  If $z-w$ is in the support of $\chi_j$ for $z,w\in U$, then $\abs{z-w}\leq j^{-1}$, so $\abs{\rho(z)-\rho(w)}\leq\abs{z-w}\leq j^{-1}$ (see   \cite[Theorem~4.8(1)]{Fed59}).  Hence,
\begin{equation}
\label{eq:rho_j_estimate}
  \abs{\rho(z)-\tilde\rho_j(z)}=\abs{\int_{\mathbb{C}^n}(\rho(z)-\rho(w))\chi_j(z-w)dV_w}\leq j^{-1}.
\end{equation}
Using similar reasoning, there exists $C>0$ such that for $j$ sufficiently large and $z\in U_j$, we have
\begin{equation}
\label{eq:grad_rho_j_estimate}
  \abs{\nabla\rho(z)-\nabla\tilde\rho_j(z)}\leq C j^{-1}.
\end{equation}
Since $U$ is an open neighborhood of $b{D}$, there must exist $\eps>0$ such that the pre-image of the interval $[-\eps,\eps]$ under $\rho$ is compact in $U$.  For $j>\eps^{-1}$, we can use \eqref{eq:rho_j_estimate} to show $\tilde\rho_j(z)>0$ when $\rho(z)=\eps$ and $\tilde\rho_j(z)<0$ when $\rho(z)=-\eps$, and for $j>C$ we can use \eqref{eq:grad_rho_j_estimate} to show that $\nabla\tilde\rho_j\neq 0$ on $U_j$.  Hence, for $j$ sufficiently large $\tilde\rho_j$ is a smooth defining function for a domain ${D}_j$ with smooth boundary such that $b{D}_j\subset U_j$.  On any compact $K\subset U$, $\norm{\tilde\rho_j-\rho}_{\mathcal{C}^2(K)}\rightarrow 0$.  Hence, if $\rho_j$ is the signed distance function for ${D}_j$, then we will also have $\norm{\tilde\rho_j-\rho_j}_{\mathcal{C}^2(K)}\rightarrow 0$.  Hence, if $\tau_j$ and $\eta_j$ are the natural generalizations of \eqref{eq-tracelevi} and \eqref{eq-eta_squared}, we will have $\eta_j\rightarrow\eta$ and $\tau_j\rightarrow\tau$.  Therefore, we may apply Lemma~\ref{lem:int_by_parts_gradient} to each ${D}_j$ and take the limit as $j\rightarrow\infty$ to obtain Lemma~\ref{lem:int_by_parts_gradient} when ${D}$ has $\mathcal{C}^2$ boundary.

Henceforth, we assume that ${D}$ has $\mathcal{C}^3$ boundary and that $\rho$ is $\mathcal{C}^3$ on $U$.  This will be needed, e.g., in \eqref{eq:L_j^*_L_j_commutator}, which requires three derivatives of the defining function.

\subsection{Elementary observations }

Note that from \eqref{eq-drho} we have
\begin{equation}\label{eq-rhojsquare}
 \sum_{\ell=1}^n \abs{\rho_\ell}^2 = \sum_{\ell=1}^n \abs{\rho_{\bar{\ell}}}^2 = \frac{1}{4}.
 \end{equation}
Differentiating with respect to  $z_j$, we see that for each $j=1,\dots, n$, we have
\begin{equation} \label{eq-diff1}
\sum_{\ell=1}^n \left( \rho_{\ell} \rho_{\ol{\ell} j} + \rho_{\ell j} \rho_{\ol{\ell}}\right)=0.
\end{equation}

With $N$ defined by \eqref{eq-Ndef}, we have
\[ N\rho =1, \text{ and } \norm{N}= \sqrt{2}.\]
With $L_j$ defined by \eqref{eq-Ljdef}, $L_j\rho = \rho_j - \rho_j N\rho = 0$, i.e. $L_j$ is a $(1,0)$-vector field tangent to the boundary.

The $L_j$'s are not  linearly independent. In fact
\begin{equation}\label{eq-rhojLj}
\sum_{j=1}^n \rho_{\bar j} L_j=0.
\end{equation}

\subsection{Pointwise computations preliminary to the proof of Lemma~\ref{lem:int_by_parts_gradient}}

Our primary goal in this section is to compute the sum of commutators $\sum_{j=1}^n[\ol L_j^*,\ol{L_j}]$.  This will arise naturally in the integration by parts argument needed to prove Lemma~\ref{lem:int_by_parts_gradient}.  Because the computation is somewhat lengthy, we will break it down into several pieces.

We begin by computing representation formulas for the adjoints of $N$ and $L_j$.  Recall that $X^*$ denotes the adjoint of a differential operator $X$ in the weighted spaces $L^2({D}, e^{-\phi})$.  We will use $X^{*,0}$ to denote the adjoint in the unweighted
spaces $L^2({D})$, i.e., when $\phi\equiv 0$.  Also, for a function $\psi$ we will denote by
$ \psi \Cdot$
the operator that multiplies a function or form by $\psi$.

We first note that
\[
N^*= e^{\phi}N^{*,0}(e^{-\phi}\,{\Cdot})
= - e^{\phi}\sum_{k=1}^n\ \frac{\partial}{\partial \ol{z_k}}
 (4\rho_k  e^{-\phi}\,\Cdot ),
\]
so
\begin{equation}
N^*= -\ol{N}  - 4 \nu\, \Cdot \label{eq-nstar}
\end{equation}
where
\begin{equation}
\label{eq-nu}
\nu =\sum_{k=1}^n (\rho_{k\ol{k}}- \rho_k\phi_{\ol{k}}).
\end{equation}

Since $\ol{L_j} = \frac{\partial}{\partial \ol{z_j}} - \rho_j \ol{N}$, we have
\[ \overline{L}_j^{*} =\left(\frac{\partial}{\partial \ol{z_j}}\right)^{*} - \ol{N}^{*}(\rho_j \Cdot)\]
We know
\[
  \left(\frac{\partial}{\partial \ol{z_j}}\right)^{*}=-e^{\phi}\frac{\partial}{\partial z_j}(e^{-\phi}\,\Cdot)=- \frac{\partial}{\partial z_j} +\phi_j \Cdot
\]
so combining this with \eqref{eq-nstar} we obtain
\begin{align*} \overline{L}_j^{*} &= - \frac{\partial}{\partial z_j} +\phi_j \Cdot + (N+4\ol{\nu}) (\rho_j \Cdot)\\
&= - \left( \frac{\partial}{\partial z_j} - \rho_j N\right) + \phi_j \Cdot+  N (\rho_j \Cdot) + 4\ol{\nu}\rho_j \Cdot\\
&= - L_j  + \phi_j\Cdot+  4 \left( \sum_{k=1}^n \rho_{\ol{k}}\rho_{jk} + \rho_j \ol{\nu} \right) \Cdot\end{align*}
Therefore  we obtain the representation
\begin{equation}\label{eq-Ljstar}
\ol{L_j}^* = -L_j + \mu_j\, \Cdot
\end{equation}
where $\mu_j$ is given by
\begin{equation}\label{eq-muj}
\mu_j=\phi_j+ 4\sum_{k=1}^n \rho_{jk}\rho_{\ol{k}}+ 4\rho_j \ol{\nu}.
\end{equation}
Note that thanks to \eqref{eq-diff1} applied to the middle term in \eqref{eq-muj}, we have the
alternative expression involving the complex Hessian for $\rho$:
\begin{equation}\label{eq-mujbis}
\mu_j=\phi_j- 4\sum_{k=1}^n \rho_{j\ol{k}}\rho_{{k}}+ 4\rho_j \ol{\nu}.
\end{equation}

Returning to our goal for this section, we use \eqref{eq-Ljstar} to show that
\begin{align}
  \sum_{j=1}^n[\ol L_j^*,\ol{L_j}]&=\sum_{j=1}^n[-L_j + \mu_j\Cdot,\ol{L_j}]\nonumber\\
  &=-\sum_{j=1}^n[L_j,\ol{L_j}]-\sum_{j=1}^n\ol{L_j}\mu_j\label{eq:L_j^*_L_j_commutator}
\end{align}
Thus, there are two terms that we need to compute in order to find this commutator.

To evaluate \eqref{eq:L_j^*_L_j_commutator}, we first consider $\sum_{j=1}^n[L_j,\ol{L_j}]$.  We compute
\begin{align*}
\sum_{j=1}^n [L_j, \ol{L_j}] &=\sum_{j=1}^n\left \{ L_j \left( \frac{\partial}{\partial \ol{z_j}} - \rho_{\ol{j}}\ol{N}\right) - \ol{L_j} \left( \frac{\partial}{\partial z_j}- \rho_{j} {N}\right) \right\}\\
&= \sum_{j=1}^n\left \{ L_j  \frac{\partial}{\partial \ol{z_j}} - \ol{L_j}  \frac{\partial}{\partial z_j} \right\}
+\sum_{j=1}^n\left\{- L_j  \rho_{\ol{j}}\ol{N} - \rho_{\ol{j}}L_j \ol{N}  + \ol{L_j}\rho_{j} {N} +\rho_j \ol{L_j}N \right\}.
\end{align*}
Using \eqref{eq-rhojLj}, we have
\[
\sum_{j=1}^n [L_j, \ol{L_j}] = \sum_{j=1}^n\left \{ L_j  \frac{\partial}{\partial \ol{z_j}} - \ol{L_j}  \frac{\partial}{\partial z_j} \right\}
+\sum_{j=1}^n\left\{- L_j  \rho_{\ol{j}}\ol{N}  + \ol{L_j}\rho_{j} {N} \right\}.
\]
Since $\sum_{j=1}^n \left( -L_j\rho_{\ol{j}} \ol{N} + \ol{L_j}\rho_j N\right)= \tau (N-\ol{N})$ and
\[
\sum_{j=1}^n\left \{ L_j  \frac{\partial}{\partial \ol{z_j}} - \ol{L_j}  \frac{\partial}{\partial z_j} \right\}= \sum_{j=1}^n\left[ \frac{\partial}{\partial z_j}, \frac{\partial}{\partial\ol{ z_j}}\right] -  4\sum_{j,k=1}^n\rho_j\rho_{\ol{k}}  \frac{\partial^2}{\partial z_k\partial \ol{z_j}} +  4\sum_{j,k=1}^n\rho_{\ol{j}}\rho_{{k}}  \frac{\partial^2}{\partial \ol{z_k}\partial{z_j}}
=0,
\]
we obtain
\begin{equation}\label{eq-commllbar}\sum_{j=1}^n [L_j, \ol{L_j}] = \tau (N-\ol{N}).\end{equation}

Before turning our attention to $\sum_{j=1}^n \ol{L_j}\mu_j$ in \eqref{eq:L_j^*_L_j_commutator}, we will find it helpful to compute $N\tau$.  We observe that
\begin{align*}
N\tau &= 4 \sum_{j,k=1}^n \rho_{\bar{k}} \frac{\partial}{\partial z_k} \left( \rho_{j \ol{j}} - \sum_{\ell=1}^n 4 \rho_{\ol{j}} \rho_{\ell}\rho_{j\ol{\ell}}\right)\nonumber\\
&= 4 \sum_{j,k=1}^n \left\{ \rho_{\ol{k}} \rho_{j \ol{j}k} - 4 \rho_{\ol{k}} \sum_{\ell=1}^n \left( \rho_{\ol{j}k} \rho_{\ell}\rho_{j \ol{\ell}} + \rho_{\ol{j}}\rho_{\ell k} \rho_{j \ol{\ell}} + \rho_{\ol{j}} \rho_{\ell} \rho_{j \ol{\ell}k} \right) \right\}.
\end{align*}
If we re-index and apply \eqref{eq-diff1}, we see that
\[
  4 \sum_{j,k,\ell=1}^n \left\{- 4 \rho_{\ol{k}} \left( \rho_{\ol{j}k} \rho_{\ell}\rho_{j \ol{\ell}} + \rho_{\ol{j}}\rho_{\ell k} \rho_{j \ol{\ell}}\right) \right\}=-16 \sum_{j,k,\ell=1}^n \left\{ \left( \rho_{\ol{k}}\rho_{\ol{j}k} \rho_{\ell}\rho_{j \ol{\ell}} + \rho_{\ol{\ell}}\rho_{\ol{k}}\rho_{j \ell} \rho_{k \ol{j}}\right) \right\}=0,
\]
so
\begin{equation}
\label{eq-ntau}
N\tau =  4 \sum_{j,k=1}^n  \rho_{\ol{k}} \rho_{j \ol{j}k} - 16 \sum_{j,k,\ell=1}^n \rho_{\ol{k}} \rho_{\ol{j}} \rho_{\ell} \rho_{j \ol{\ell}k}.
\end{equation}

Now we are ready to compute $\sum_{j=1}^n \ol{L_j}\mu_j$, thus completing our computation of \eqref{eq:L_j^*_L_j_commutator}.  We will show:
\begin{equation}\label{eq-ljmuj}
\sum_{j=1}^n \ol{L_j}\mu_j = \gamma_\phi + N\tau+ 4 \eta^2 +4 \tau\ol{\nu}.
\end{equation}
To show this, we first compute
\[
\sum_{j=1}^n \ol{L_j}\mu_j = \sum_{j=1}^n \ol{L_j}\phi_j+4 \sum_{j,k=1}^n \left(\ol{L_j} \rho_{jk}\rho_{\ol{k}}+ \rho_{jk}\ol{L_j} \rho_{\ol{k}}\right) + 4\sum_{j=1}^n\left(\ol{L_j} \rho_j \ol{\nu} + \rho_j \ol{L_j} \ol{\nu} \right)
\]
Thanks to \eqref{eq-rhojLj}, the final term in this expression vanishes when summed over $j$.  Using the notation of \eqref{eq-tracelevi}, we may rewrite this as:
\begin{equation}
\sum_{j=1}^n \ol{L_j}\mu_j = \gamma_\phi+4 \sum_{j,k=1}^n \ol{L_j} \rho_{jk}\rho_{\ol{k}}+ 4 \sum_{j,k=1}^n\rho_{jk}\ol{L_j} \rho_{\ol{k}} +4 \tau\ol{\nu}. \label{eq-ljmuj2}
\end{equation}
The second term in \eqref{eq-ljmuj2} may be rewritten as
\[
4 \sum_{j,k=1}^n \ol{L_j} \rho_{jk}\rho_{\ol{k}} =   4 \sum_{j,k=1}^n  \rho_{\ol{k}} \rho_{j \ol{j}k} - 16 \sum_{j,k,\ell=1}^n \rho_{\ol{k}} \rho_{\ol{j}} \rho_{\ell} \rho_{j \ol{\ell}k}.
\]
Using \eqref{eq-ntau}, we have
\begin{equation}
4 \sum_{j,k=1}^n \ol{L_j} \rho_{jk}\rho_{\ol{k}} =N\tau \label{eq-ntau2}.
\end{equation}
Using the fact that $\rho_{jk} = L_j \rho_k + \rho_j N \rho_k$, the third term in \eqref{eq-ljmuj2} may be rewritten as
\[
\sum_{j,k=1}^n\rho_{jk}\ol{L_j} \rho_{\ol{k}}= \sum_{j,k=1}^nL_j\rho_k  \ol{L_j} \rho_{\ol{k}} + \sum_{j,k=1}^n\rho_j N \rho_k \ol{L_j} \rho_{\ol{k}},
\]
but the final term in this expression vanishes by \eqref{eq-rhojLj}, so from \eqref{eq-eta_squared} we have
\begin{equation}
\sum_{j,k=1}^n\rho_{jk}\ol{L_j} \rho_{\ol{k}}= \eta^2. \label{eq-eta}
\end{equation}
Putting \eqref{eq-ntau2} and \eqref{eq-eta} in \eqref{eq-ljmuj2}, the result \eqref{eq-ljmuj}	follows.

Substituting \eqref{eq-commllbar} and \eqref{eq-ljmuj} into \eqref{eq:L_j^*_L_j_commutator}, we have the fundamental identity
\[
  \sum_{j=1}^n[\ol L_j^*,\ol{L_j}]=-\tau (N-\ol{N})-\left(\gamma_\phi + N\tau+ 4 \eta^2 +4 \tau\ol{\nu}\right)
\]
Using \eqref{eq-nstar}, we can rewrite this as
\begin{equation}
  \sum_{j=1}^n[\ol L_j^*,\ol{L_j}]=(\tau\ol{N})^*+\tau\ol{N}-\gamma_\phi - 4 \eta^2\label{eq:commlbar^*lbar}
\end{equation}

\subsection{Proof of Lemma~\ref{lem:int_by_parts_gradient}}

Let $v$ be a smooth function that is compactly supported in $U$.     Since the $L_j$'s are tangential, we can integrate by parts with no boundary term and obtain
\[
  \sum_{j=1}^n \norm{\ol{L_j}v}_{{\phi}}^2= \gipr{\left(\sum_{j=1}^n [\ol{L_j}^*,\ol{L_j}]\right)v,v}_{{\phi}}  +\sum_{j=1}^n\norm{\ol{L_j}^* v}^2_{{\phi}}
\]
Using \eqref{eq:commlbar^*lbar}, we have
\begin{equation}
\label{eq:precommutator}
  \sum_{j=1}^n \norm{\ol{L_j}v}_{{\phi}}^2= \gipr{\left((\tau\ol{N})^*+\tau\ol{N}-\gamma_\phi - 4 \eta^2\right)v,v}_{{\phi}}  +\sum_{j=1}^n\norm{\ol{L_j}^* v}^2_{{\phi}}
\end{equation}
When we integrate by parts with $(\tau\ol{N})^*$, we will need to consider a boundary term.  Since $-\rho$ is a normalized defining function for the complement of ${D}$ and $v$ is compactly supported in $U$, we have
\[
  \gipr{v,\tau\ol{N}v}_{{\phi}}=\gipr{(\tau\ol{N})^*v,v}_{{\phi}}+\int_{b(U{\setminus}\ol{D})}\tau\abs{v}^2 N(-\rho) e^{-\phi} d\sigma.
\]
Since $N\rho=1$ and $\int_{b(U{\setminus}\ol{D})} \tau\abs{v}^2 e^{-\phi} d\sigma=\int_{b{D}} \tau\abs{v}^2 e^{-\phi} d\sigma$, we have
\[
  \gipr{v,\tau\ol{N}v}_{{\phi}}=\gipr{(\tau\ol{N})^*v,v}_{{\phi}}-\int_{b{D}}\tau\abs{v}^2 e^{-\phi} d\sigma.
\]
Substituting in \eqref{eq:precommutator}, we obtain
\begin{multline}
  \sum_{j=1}^n \norm{\ol{L_j}v}_{{\phi}}^2=\int_{b{D}} \tau\abs{v}^2 e^{-\phi} d\sigma+ 2\Re\gipr{\tau\ol{N}v,v}_{{\phi}}\\
  -\gipr{\left(\gamma_\phi + 4 \eta^2\right)v,v}_{{\phi}}  +\sum_{j=1}^n\norm{\ol{L_j}^* v}^2_{{\phi}} \label{eq-commutator}.
\end{multline}

We now wish to decompose $\ol{L_j}^*$ into its tangential and normal components.  Using \eqref{eq-Ljstar}, we have
\[
 \sum_{j=1}^n \rho_{\bar j}\ol{L_j}^* = \sum_{j=1}^n \rho_{\ol{j}}(-L_j+\mu_j).
 \]
 Therefore, \eqref{eq-rhojLj} implies
\[
 \sum_{j=1}^n \rho_{\bar j}\ol{L_j}^* =\sum_{j=1}^n\rho_{\bar j}\mu_j.
 \]
 Substituting \eqref{eq-mujbis} and \eqref{eq-rhojsquare} gives us
\[
 \sum_{j=1}^n \rho_{\bar j}\ol{L_j}^* = \sum_{j=1}^n \rho_{\ol{j}}\phi_j - 4 \sum_{j,k=1}^n \rho_{\ol{j}}\rho_{j\ol{k}}\rho_{k} + \ol{\nu}, \]
 and substituting \eqref{eq-nu} gives us
\[
 \sum_{j=1}^n \rho_{\bar j}\ol{L_j}^* = \sum_{j=1}^n \rho_{\ol{j}j} - 4 \sum_{j,k=1}^n \rho_{\ol{j}}\rho_{j\ol{k}}\rho_{k}= \sum_{j=1}^n \ol{L_j}\rho_j.
 \]
Hence, we have
\begin{equation}\label{eq-rhobarjLjstar}
  \sum_{j=1}^n \rho_{\bar j}\ol{L_j}^*=\tau.
\end{equation}

By orthogonal decomposition, we have
\[
  \sum_{j=1}^n \abs{\ol{L_j}^* v}^2=\sum_{j=1}^n \abs{\left(\ol{L_j}^*-4\rho_j\sum_{k=1}^n\rho_{\bar k}\ol{L_k}^*\right) v}^2+4\abs{\sum_{j=1}^n\rho_{\bar j}\ol{L_j}^* v}^2,
\]
so \eqref{eq-rhobarjLjstar} gives us
\[
  \sum_{j=1}^n \abs{\ol{L_j}^* v}^2= \sum_{j=1}^n \abs{(\ol{L_j}^*-4\rho_j\tau) v}^2+4\abs{\tau v}^2.
\]
Therefore, we have
\begin{equation}\label{eq-new1}
\sum_{j=1}^n\norm{\ol{L_j}^* v}^2_{{\phi}}   =  \sum_{j=1}^n \norm{(\ol{L_j}^*-4\rho_j\tau) v}_{{\phi}}^2+4\norm{\tau v}_{{\phi}}^2.
\end{equation}
Combining \eqref{eq-commutator} and \eqref{eq-new1} we obtain \eqref{eq:int_by_parts_gradient}, and Lemma~\ref{lem:int_by_parts_gradient} is proven.

\section{A Modified Morrey-Kohn-H\"ormander Identity}
\label{sec:main_identity}
For a form $u = \sum_{\abs{J}=p,\abs{K}=q}' u_{J,K}dz^J\wedge d\bar{z}^K$, we denote by $\ol{\nabla}u$ the $(0,1)$-part of
the covariant derivative of $u$, i.e., the tensor
\[ \ol{\nabla} u = {\sideset{}{'}\sum_{\abs{J}=p,\abs{K}=q}}\,\sum_{j=1}^n\frac{\partial u_{J,K}}{\partial \ol{z_j}} dz^J\wedge d\ol{z}_K\tensor d\ol{z}_j.\]
In order to prove Theorem \ref{thm:pseudoconcave_identity}, we will first show that
with the same notation as Theorem \ref{thm:pseudoconcave_identity}, we have
\begin{lem}\label{lem:gradient_identity}
\begin{align}
\norm{\ol{\nabla} u}_{{\phi}}^2& = \int_{bD} \tau \abs{u}^2 e^{-\phi} d\sigma
 -\gipr{\gamma_\phi u,u}_{{\phi}}
  -4\norm{\eta u}_{{\phi}}^2\nonumber\\
&+\sum_{j=1}^n \norm{(\ol{L_j}^*-4\rho_j\tau) u}_{{\phi}}^2 + \frac{1}{4}  \norm{(\ol{N} +4\tau)u}_{{\phi}}^2\label{eq-twistedgradient}
\end{align}
\end{lem}

\begin{proof}
We first consider the special case in which $v\in \mathcal{C}^1(U{\setminus}D)$ with compact support in  $U{\setminus}D$.

Rewrite \eqref{eq-Ljdef} as $ \frac{\partial}{\partial z_j}= L_j + \rho_j N$, so $\frac{\partial}{\partial\ol{z_j}}=
\ol{L_j} + \rho_{\ol{j}}\ol{N}$.  From \eqref{eq-rhojsquare} and \eqref{eq-rhojLj} we have the orthogonal decomposition
\[\sum_{j=1}^n \abs{ v_{\bar{j}}}^2
=  \sum_{j=1}^n  \abs{\ol{L_j} v}^2 + \frac{1}{4}\abs{\ol{N}{v}}^2.
\]
Therefore,
\[
\norm{\ol{\nabla}v}_{{\phi}}^2 = \sum_{j=1}^n\norm{\ol{L_j}v}^2_{{\phi}} + \frac{1}{4} \norm{\ol{N}v}_{{\phi}}^2.
\]
By Lemma \ref{lem:int_by_parts_gradient}, we have
\begin{align*}
  \norm{\ol{\nabla}v}_{{\phi}}^2&=\sum_{j=1}^n \norm{(\ol{L_j}^*-4\rho_j\tau) v}_{{\phi}}^2+  4\norm{\tau v}_{{\phi}}^2
  + \int_{b\Omega} \tau\abs{v}^2 e^{-\phi} d\sigma  + 2 \Re \gipr{\tau v, \ol{N} v}_{{\phi}} \\ &\qquad-\gipr{\left(\gamma_\phi + 4 \eta^2\right)v,v}_{{\phi}}+ \frac{1}{4} \norm{\ol{N}v}_{{\phi}}^2.
\end{align*}
Note that
\[ \norm{\ol{N}v}_{{\phi}}^2+ 2 \Re \gipr{4\tau v, \ol{N} v}_{{\phi}}= \norm{(\ol{N} +4\tau)v}_{{\phi}}^2 - 16 \norm{\tau v}_{{\phi}}^2. \]
Therefore we obtain
\begin{multline}
\label{eq-ibp}
  \norm{\ol{\nabla}v}_{{\phi}}^2=\sum_{j=1}^n \norm{(\ol{L_j}^*-4\rho_j\tau) v}_{{\phi}}^2
   +\int_{b\Omega} \tau\abs{v}^2 e^{-\phi} d\sigma \\ -\gipr{\left(\gamma_\phi + 4 \eta^2\right)v,v}_{{\phi}}+ \frac{1}{4} \norm{(\ol{N} +4\tau)v}_{{\phi}}^2.
\end{multline}
which, after rearrangement, is identical to \eqref{eq-twistedgradient} for functions.

Now let $u=\sum_{\abs{I}=p,\abs{J}=q}'u_{I,J} dz^I\wedge d \ol{z}^J$
 be as in the statement of the proposition. Fix  increasing multi-indices $I$ and
$J$ with $\abs{I}=p$ and $\abs{J}=q$. Replacing $v$ in \eqref{eq-ibp} by $u_{I,J}$ and, summing over all $I,J$ with $\abs{I}=p,\abs{J}=q$, we obtain \eqref{eq-twistedgradient}.
\end{proof}
\subsection{Proof of Theorem~\ref{thm:pseudoconcave_identity}}
We recall the Morrey-Kohn-H\"ormander Identity of $L^2$-theory: Let $\Omega$ be a domain with $\mathcal{C}^2$ boundary, for $1\leq q \leq n, 0\leq p\leq n$ let $u$ be a $(p,q)$-form on $\overline{\Omega}$ such that
\begin{equation}\label{eq-u}
u\in {\dom}(\dbar^*)\cap \mathcal{C}^1(\overline{\Omega}),
\end{equation}
and let $\phi\in \mathcal{C}^2(\ol{\Omega})$ with $\phi\in \rl$. Let $r$ be a normalized defining function for $\Omega$, i.e., $\abs{dr}=1$ along $b\Omega$.
Then we have
\begin{equation}
\norm{\dbar u}_{L^2(\Omega, e^{-\phi})}^2+ \norm{\dbar^*_\phi u}^2_{L^2(\Omega, e^{-\phi})}=  \int_{b\Omega}H r(u,u) e^{-\phi}d\sigma
 +\int_\Omega(H\phi)(u,u) e^{-\phi} dV
 + \norm{\ol{\nabla}u}^2_{L^2(\Omega, e^{-\phi})}.\label{eq-twistedbasicidentity}
\end{equation}
See  \cite[Proposition 4.3.1]{ChSh01} or \cite[page 18 ff.]{Str10} for a detailed account of the proof of this basic result.

Recall  that $U$ is a bounded neighborhood of  $\ol{D}$. Choose $R>0$ sufficiently large so that  $\ol{U}\subset B(0,R)$, and let
 $\Omega=B(0,R)\setminus \ol{D}$. Recall that the form $u$ has compact support in $U\setminus D$ and lies in $\mathcal{C}^1_{p,q}(\ol{U}\setminus D)\cap\dom \dbar^*_\phi$.
 Since $b\Omega$ is the disjoint union of $bD$ and $bB(0,R)$, it follows that $u|_{b B(0,R)}=0$, and  we may apply \eqref{eq-twistedbasicidentity} to this domain to obtain
\begin{multline*}
\norm{\dbar u}_{L^2(U{\setminus}\ol{D},e^{-\phi})}^2+ \norm{\dbar^*_\phi u}^2_{L^2(U{\setminus}\ol{D},e^{-\phi})} =  \\ \int_{b(U{\setminus}\ol{D})}H(-\rho)(u,u) e^{-\phi}d\sigma
 +\int_{U{\setminus}\ol{D}}(H\phi)(u,u) e^{-\phi} dV
 + \norm{\ol{\nabla}u}^2_{L^2(U{\setminus}\ol{D},e^{-\phi})}.
\end{multline*}
where  we have used the fact that $-\rho$ is a normalized defining function for $\Omega$ near $bD$.  Since \[\int_{b(U{\setminus}\ol {D})}H(-\rho)(u,u) e^{-\phi}d\sigma=-\int_{bD}H\rho(u,u) e^{-\phi}d\sigma,\] we may combine the above identity with \eqref{eq-twistedgradient} to obtain \eqref{eq:pseudoconcave_identity}.

\section{Proof of Theorem~\ref{thm:W^1_solvability}}
\label{sec:mixed_boundary_conditions}
\subsection{Boundary conditions in the \texorpdfstring{$\dbar$}{dbar}-problem} In the $L^2$-theory of the $\dbar$-operator, it is standard to use the {\em maximal weak realization} of the
 $\dbar$-operator as a  densely-defined closed  unbounded Hilbert space operator on the weighted spaces of  square-integrable
 forms $L^2_{p,q}(D, e^{-\phi})$.  The operator $\dbar:L^2_{p,q}(D, e^{-\phi})\to L^2_{p,q+1}(D, e^{-\phi})$ has dense domain
 \[ \dom(\dbar)= \left\{u\in L^2_{p,q}(D, e^{-\phi})| \dbar u \in  L^2_{p,q+1}(D, e^{-\phi})\right\}.\]
It will be useful for our purposes to consider other ways of realizing the $\dbar$-operator as
an unbounded operator on $L^2_{p,q}(D, e^{-\phi})$.
We will find it helpful to use the strong minimal realization of the $\dbar$-operator, defined as follows:
\begin{defn}
\label{defn:strong_minimal} Let $D$ be an open set  in $\cx^n$. Let $u$ be a $(p,q)$-form on $D$.  We say that $u\in\dom(\dbar_c)$ if there exists a sequence $\{u_j\}\subset L^2_{p,q}(D)\cap\dom\dbar$ and $v\in L^2_{p,q+1}(D)$ such that each  $u_j$ is compactly supported in $D$,  we have $u_j\rightarrow u$ in $L^2_{p,q}(D)$, and  $\dbar u_j\rightarrow v$ in $L^2_{p,q+1}(D)$.  When dealing with multiple domains, we will write $\dbar_c=\dbar_c^D$ to emphasize the domain under consideration.
\end{defn}
As noted in \cite{LiSh13}, this means that $u$ must satisfy the $\dbar$-Dirichlet condition on $b\Omega$, i.e., if $u\in \mathcal{C}^1_{p,q}(\overline\Omega)$, then $u\in\dom(\dbar_c)$ if and only if $\dbar\rho\wedge u|_{b\Omega}=0$ for any $\mathcal{C}^1$ defining function $\rho$ for $\Omega$.
We will need  the following fact, for whose proof see  \cite[Lemma 2.4]{LaSh13}:
\begin{lem}[Laurent-Thi\'ebaut, Shaw]
\label{lem:extension_str_min}
  Let $D$ be a bounded domain in $\cx^n$ and let $u$ be a $(p,q)$-form on $D$.  If $u\in\dom(\dbar_c^D)$, then the extension $\tilde u\in L^2_{p,q}(\mathbb{C}^n)$ defined by $\tilde u=u$ on $D$ and $\tilde u=0$ on $\mathbb{C}^n{\setminus} D$ is in $\dom(\dbar)$ on $\mathbb{C}^n$.  The converse is also true when $D$ has Lipschitz boundary.
\end{lem}
We will also need to make use of the $\dbar$-operator with mixed boundary conditions on annuli (see Li and Shaw   \cite{LiSh13}):
\begin{defn}
\label{defn:mix} Let $\Omega$ be a bounded annulus with envelope $\Omega_1$ and hole $\Omega_2$. Let $u$ be a $(p,q)$-form on $\Omega$.  We say that $u\in\dom(\dbar_{\mathrm{mix}})$ if there exists a sequence $\{u_j\}\subset L^2_{p,q}(\Omega)\cap\dom\dbar$ and $v\in L^2_{p,q+1}(\Omega)$ such that $u_j\rightarrow u$, $\dbar u_j\rightarrow v$, and $u_j$ vanishes identically in a neighborhood of $b\Omega_2$.
\end{defn}
As  for $\dom(\dbar_c)$, this means that $u$ must satisfy the $\dbar$-Dirichlet condition on $b\Omega_2$, i.e., if $u\in \mathcal{C}^1_{p,q}(\overline\Omega)$, then $u\in\dom(\dbar_{\mathrm{mix}})$ if and only if $\dbar\rho\wedge u|_{b\Omega_2}=0$ for any $\mathcal{C}^1$ defining function $\rho$ for $\Omega$. We will frequently make use of the following, which follows using the same argument (cf. \cite[ Lemma 2.4]{LaSh13}) as for Lemma~\ref{lem:extension_str_min}
:
\begin{lem}
\label{lem:extension}
  Let $\Omega$ be a bounded annulus with envelope $\Omega_1$ and hole $\Omega_2$, and let $u$ be a $(p,q)$-form on $\Omega$.  If $u\in\dom(\dbar_{\mathrm{mix}})$, then the extension $\tilde u\in L^2_{p,q}(\Omega)$ defined by $\tilde u=u$ on $\Omega$ and $\tilde u=0$ on $\ol\Omega_2$ is also in $\dom(\dbar)$ for $\Omega_1$.  The converse is also true when $\Omega_2$ has Lipschitz boundary.
\end{lem}
For a general discussion of the Hilbert space realization of differential operators, see
\cite{grubb}, and in the context of the $\dbar$-operator, see \cite{ChSh12}.

\subsection{Solvability with mixed boundary conditions: statement and preliminaries}
The work of this section, leading to a proof of  Theorem~\ref{thm:W^1_solvability}, follows closely that of \cite{LiSh13}. Like in \cite{LiSh13}, we exploit the close relation between
the $W^1$-estimates on the hole for the $\dbar$-problem, and estimates on the $\dbar$-problem
on the annulus with mixed boundary conditions. In particular, along with our proof of
Theorem~\ref{thm:W^1_solvability}, we will obtain the following result:

\begin{prop}
\label{prop:dbar_mix_solvability}
For $n\geq 2$ and $2\leq q\leq n$, let $\Omega$ be an annulus in $\cx^n$, such that the envelope is weakly $q$-convex and the hole is  weakly $(q-1)$-convex with $\mathcal{C}^{1,1}$ boundary.
Then  there exists a constant $C_q(\Omega,\mathrm{mix})>0$ with the property that for all $0\leq p\leq n$ and for every $f\in {\dom}(\dbar_{\rm mix}) \cap L^2_{p,q}(\Omega)$ satisfying $\dbar f =0$ there is a $u\in  {\dom}(\dbar_{\rm mix}) \cap L^2_{p,q-1}(\Omega)$ such that $\dbar_{\rm mix} u =f$ and
\[
  \norm{u}_{L^2(\Omega)}\leq C_q(\Omega,\mathrm{mix})\norm{f}_{L^2(\Omega)}.
\]
\end{prop}
In \eqref{eq:mixed_dbar_constant}, we will compute an explicit upper bound for $C_q(\Omega,\mathrm{mix})$ in terms of the geometry of $\Omega$.

Proposition \ref{prop:dbar_mix_solvability} is a generalization  of  \cite[Theorem~2.2]{LiSh13}. We have relaxed the convexity requirement from pseudoconvex to weakly $q$-convex
{ on the envelope and from pseudoconvex to weakly $(q-1)$-convex on the hole}, and  we have reduced the regularity of the hole from $\mathcal{C}^2$ to $\mathcal{C}^{1,1}$.

In order to apply Theorem \ref{thm:pseudoconcave_identity} to weakly $q$-convex domains, we will need some simple estimates for the Levi-form.
\begin{lem}
\label{lem:Levi-form_estimates}
  Let $\Omega\subset\mathbb{C}^n$ be a domain with $\mathcal{C}^2$ boundary, and let $\rho$ be a normalized defining function for $\Omega$.  For $0\leq p\leq n$ and $1\leq q\leq n$, let $u\in\mathcal{C}^1_{(p,q)}(\ol\Omega)\cap\dom\dbar^*$.  Then we have
  \begin{equation}
  \label{eq:weak_q_convexity}
    H\rho(u,u)|_{b\Omega}\geq 0\text{ whenever }\Omega\text{ is weakly }q\text{-convex}.
  \end{equation}
  If $n\geq 3$ and $1\leq q\leq n-2$, then
  \begin{equation}
  \label{eq:weak_n-1-q_convexity}
    (H\rho-\tau I)(u,u)|_{b\Omega}\leq 0\text{ whenever }\Omega\text{ is weakly }(n-1-q)\text{-convex}.
  \end{equation}
\end{lem}

\begin{proof}

Fix $w\in b\Omega$.  After a unitary change of coordinates,  we may assume that $\rho_j(w)=0$ for all $1\leq j\leq n-1$ and the hermitian matrix
$\left(\rho_{j\ol k}(w)\right)_{1\leq j,k\leq n-1}$ (i.e., the Levi-form) is diagonal.  As usual, we write $u=\sum_{\abs{J}=p,\abs{K}=q}'u_{J,K}dz^J\wedge d\ol{z}^K$.  Since $u\in\dom\dbar^*$, we have $u_J|_w=0$ whenever $n\in J$.  Hence, we have
\[
  H\rho(u,u)|_w={\sideset{}{'}\sum_{\abs{J}=p,|I|=q-1}}\,\sum_{j=1}^{n-1}\rho_{j\ol j}(w)|u_{J,jI}(w)|^2
\]
Given an increasing multi-index $K$ of length $q$ and an integer $j\in K$, there is a unique increasing multi-index of length $I$ such that $jI$ is a rearrangement of $K$.  Hence, we have the alternative expression
\begin{equation}
\label{eq:q_Levi_form}
  H\rho(u,u)|_w={\sideset{}{'}\sum_{\substack{\abs{J}=p,|K|=q\\n\notin K}}}\,\sum_{j\in K}\rho_{j\ol j}(w)|u_{J,K}(w)|^2
\end{equation}
If $\Omega$ is weakly $q$-convex, then $\sum_{j\in {K}}\rho_{j\ol j}>0$ for every increasing multi-index ${K}$ of length $q$ with $n\notin {K}$, so \eqref{eq:weak_q_convexity} follows.

Continuing to use our special coordinates at $w$, we use \eqref{eq-tracelevi} and the fact that $L_j|_w=\frac{\partial}{\partial z_j}$ for $1\leq j\leq n-1$ and $L_n|_w=0$ to compute
\begin{equation}
\label{eq:trace_levi_form}
  \tau|_w=\sum_{j=1}^{n-1}\rho_{j\ol j}(w).
\end{equation}
Hence, using \eqref{eq:q_Levi_form}, we have
\[
  (H\rho-\tau I)(u,u)|_w=-{\sideset{}{'}\sum_{\substack{\abs{J}=p,|K|=q\\n\notin K}}}\,\sum_{\substack{1\leq j\leq n-1\\j\notin K}}\rho_{j\ol j}(w)|u_{J,K}(w)|^2,
\]
and \eqref{eq:weak_n-1-q_convexity} follows.

\end{proof}

\begin{lem}\label{lem-qconvex} Let $D\Subset \cx^n$ be a bounded weakly  $q$-convex domain, where $1\leq q \leq n$ and let $\delta =\sup_{z,z'\in D}\abs{z-z'}$ be the diameter of $D$. Let $0\leq p \leq n$. Then for each $f\in L^2_{p,q}(D)$ such that $\dbar f=0$,  there is a $u\in L^2_{p,q-1}(D)$ such that  $\dbar u=f$,
\begin{equation}\label{eq-poincare} \norm{u}_{L^2(D)} \leq \sqrt{\frac{e}{q}}\cdot \delta \cdot\norm{f}_{L^2(D)},\end{equation}
and the coefficients of $u$ are locally in the Sobolev space $W^1$.

\end{lem}
\begin{proof}
In the special case when $D$ is further assumed to be pseudoconvex (i.e. weakly $1$-convex) this is
a well-known result of Hörmander (\cite{Hor65}, see also \cite[Theorem 4.3.4]{ChSh01}), where $u$ is the so-called ``canonical solution'' of the $\dbar$-problem, i.e., the solution with minimal $L^2$-norm.
When the domain is only weakly $q$-convex, using the fact that  conclusion \eqref{eq:weak_q_convexity} in Lemma~\ref{lem:Levi-form_estimates}
suffices to prove the ``basic estimate'' (\cite[Proposition 4.3.3]{ChSh01}), this was observed in
\cite[Theorem~3.1]{Ho91}. The precise constant can be found exactly as in \cite[Theorem 4.3.4]{ChSh01}.
As in the pseudoconvex case, one first proves the existence of $u$ for $\mathcal{C}^2$ boundaries which are
weakly $q$-convex, then uses an exhaustion of $D$ by weakly $q$-convex domains (see   \cite[Theorem 2.2.1]{Hor65}). The claim about local $W^1$-regularity of the canonical solution follows
as usual from the interior elliptic regularity of the complex Laplacian.
\end{proof}

We will  need the modified Hodge star operator used in \cite{LiSh13}, which is a special case of the Hodge star
 of a  line bundle with a Hermitian metric (see \cite[Section 2.4]{ChSh12}).  Here the bundle $E$ is a trivial line bundle over a
 domain $D\subset \cx^n$ with
 Hermitian metric $e^{-\phi}$. The dual bundle $E^*$ is then also trivial, but has the dual Hermitian metric $e^{\phi}$.
 Denoting by $\Lambda^{p,q}(D)$ the space of all forms of degree $(p,q)$  on $D$, the operator
$\star_\phi:\Lambda^{p,q}(D)\rightarrow\Lambda^{n-p,n-q}(D)$ is defined pointwise by the relation
\begin{equation}\label{eq-hodge}
  \left<v,u\right>e^{-\phi}dV=v\wedge\star_\phi u
\end{equation}
for all $v\in\Lambda^{n-p,n-q}(D)$.  It is well-known that  the Hodge-star operator induces an isometry of  $L^2_{p,q}(D,E)$ (square integrable $(p,q)$-forms on $D$ with values in the bundle $E$)  with  $L^2_{p,q}(D,E^*)$.
In fact,
for $(p,q)$-forms $\Psi$ and $\Phi$, we have  the easily verified pointwise relation (see \cite[Lemma 3.2]{LiSh13})
 \begin{equation}
    \label{eq:hodge_star_isometry}
     \left \langle\star_\phi\Phi,\star_\phi\Psi\right\rangle e^\phi dV=\left\langle\Psi,\Phi\right\rangle e^{-\phi}dV,
 \end{equation}
which, when $\Psi=\Phi$, can be integrated over $D$ to obtain that $\star_\phi:L^2_{p,q}(D,E)\to L^2_{p,q}(D,E^*)$ is an isometry of Hilbert spaces:
\begin{equation}
    \label{eq:hodge_star2}
\norm{\Phi}_\phi^2 = \norm{\star_\phi \Phi}_{-\phi}^2.\end{equation}
\subsection{Proof of  Proposition~\ref{prop:dbar_mix_solvability}; the \texorpdfstring{$\mathcal{C}^2$}{C2} case}
\label{sec:proof_main_theorem}

In this section, we will prove Proposition~\ref{prop:dbar_mix_solvability} under the assumption that both the envelope  $\Omega_1$ and the hole $\Omega_2$  have $\mathcal{C}^2$ boundaries (this is also the boundary regularity assumed in \cite[Theorem~2.2]{LiSh13}).

Introduce the weight  \[ \psi(z)=t|z|^2\]
for $t>0$ to be determined later.
Let $\dbar^*_{\mathrm{mix},\psi}$ denote the Hilbert-space adjoint of $\dbar_{\mathrm{mix}}$ with respect to the inner product $L^2(\Omega,e^{-\psi})$.
  We will show that a ``basic estimate''
  holds for $\dbar_{\mathrm{mix}}$.

Let $U_2$ be a neighborhood of $\ol{\Omega_2}$ such that $\rho$ defined by \eqref{eq-rhodef} is a $\mathcal{C}^2$ function on $U_2\setminus \ol{\Omega_2}$. Let $\chi\in \mathcal{C}^\infty_0(U_2)$ satisfy $\chi\equiv 1$ in a neighborhood of $\ol{\Omega_2}$.

Suppose $u\in\mathcal{C}^1_{(p,q)}(\overline\Omega)\cap\dom(\dbar_{\mathrm{mix}})\cap\dom(\dbar^*_{\mathrm{mix},\psi})$.
Define  a $(p,q)$-form $u_2$ on $\cx^n$ by setting $u_2=\chi u$ on $\ol\Omega$ and $u_2\equiv 0$ on $\ol\Omega^c$,
and  define $u_1\in \mathcal{C}^1_{0,(p,q)}(\ol\Omega_1)$ by setting $u_1=(1-\chi)u$ on $\ol\Omega$ and $u_1\equiv 0$ on $\Omega_2$.  Note that $u=u_1+u_2$ on $\ol\Omega$, and the form  $u_2$   has support in $ \ol{\Omega}$.

Now $u_1\in\mathcal{C}^1_{(p,q)}(\overline\Omega_1)\cap\dom(\dbar)\cap\dom(\dbar^*_{\psi})$, so we may apply the standard Morrey-Kohn-Hörmander estimate \eqref{eq-twistedbasicidentity} with $\phi=\psi$ to $u_1$ on $\Omega_1$.  Since $\Omega_1$ is assumed to be weakly $q$-convex, we may use \eqref{eq:weak_q_convexity} to estimate the boundary term and obtain
\[
  \norm{\dbar u_1}_\psi^2+ \norm{\dbar^*_\psi u_1}^2_\psi \geq
  \int_\Omega(H\psi)(u_1,u_1) e^{-\psi} dV,
\]
which becomes
\begin{equation}
\label{eq:key_estimate_1}
  \norm{\dbar_{\mathrm{mix}} u_1}^2_{\psi}+\norm{\dbar^*_{\mathrm{mix},\psi} u_1}^2_{\psi}\geq
  tq\norm{u_1}^2_{t|z|^2}.
\end{equation}
since  $H\psi(u_1,u_1)=tq|u_1|^2$.

Now we derive a similar estimate for the form $u_2$, using the modified Morrey-Kohn-Hörmander formula of Theorem~\ref{thm:pseudoconcave_identity},
where $\phi=-\psi$, $u=\star_\psi u_2$, and $\star_\psi$ is the Hodge star operator defined above in \eqref{eq-hodge}. We first assume that $q\leq n-1$.  Since the form $u_2$ satisfies the $\dbar$-Dirichlet boundary conditions along $b\Omega_2$, it
follows that $\star_\psi u_2$ satisfies the $\dbar$-Neumann boundary conditions (cf. \cite[Proposition~1]{ChSh12}), so that
\[ \star_{\psi} u_2\in\mathcal{C}^1_{(n-p,n-q)}(\ol\Omega)\cap\dom(\dbar)\cap\dom(\dbar^*_{-\psi}),\]
 so we can apply \eqref{eq:pseudoconcave_identity} to $\star_{\psi}u_2$ and drop the positive terms to obtain:
\begin{multline}
  \norm{\dbar(\star_{\psi}u_2)}_{-\psi}^2+ \norm{\dbar^*_{-\psi}(\star_{\psi}u_2)}^2_{-\psi} \geq  -\int_{b\Omega_2}(H\rho-\tau I)(\star_{\psi}u_2,\star_{\psi}u_2) e^{\psi}d\sigma\\
  +\int_{\Omega}\left(H(-\psi)-\gamma_{-\psi} I\right)(\star_{\psi}u_2,\star_{\psi}u_2) e^{\psi} dV-4 \norm{\eta(\star_{\psi}u_2)}_{-\psi}^2,
\label{eq:hodge_star_estimate}
\end{multline}
where the defining function $\rho$ is the signed distance function as in \eqref{eq-rhodef}.

 When $q=n$, then $\star_\psi u_2$ is a $(p,0)$-form, so $\norm{\dbar(\star_\psi u_2)}_{-\psi}=\norm{\overline\nabla(\star_\psi u_2)}_{-\psi}$.  Using Lemma \ref{lem:gradient_identity}, we have
\[
\norm{\ol{\nabla} (\star_\psi u_2)}_{-\psi}^2 \geq \int_{bD} \tau \abs{\star_\psi u_2}^2 e^{\psi} d\sigma
 -\gipr{\gamma_{-\psi} \star_\psi u_2,\star_\psi u_2}_{-\psi}
  -4\norm{\eta \star_\psi u_2}_{-\psi}^2.
\]
If we adopt the convention that $\dbar^*_{-\psi}(\star_{\psi}u_2)=0$,
\[
  (H\rho)(\star_{\psi}u_2,\star_{\psi}u_2)=0,
\]
and
\[
(H(-\psi))(\star_{\psi}u_2,\star_{\psi}u_2)=0
\]
when $\star_{\psi}u_2$ is a function, then we have precisely \eqref{eq:hodge_star_estimate}.

Since $\Omega_2$ is weakly $(q-1)$-convex and $\star_{\psi}u_2$ is an $(n-p,n-q)$-form, the conclusion \eqref{eq:weak_n-1-q_convexity} of Lemma~\ref{lem:Levi-form_estimates}
implies that $(H\rho-\tau I)(\star_{\psi}u_2,\star_{\psi}u_2)|_{b\Omega_2}\leq 0$, so we have
\begin{multline*}
  \norm{\dbar(\star_{\psi}u_2)}_{-\psi}^2+ \norm{\dbar^*_{-\psi}(\star_{\psi}u_2)}^2_{\phi} \geq\int_{\Omega}(H(-\psi)-\gamma_{-\psi} I)(\star_{\psi}u_2,\star_{\psi}u_2) e^{\psi} dV-4 \norm{\eta(\star_{\psi}u_2)}_{-\psi}^2\\
=-\int_{\Omega}(H\psi-\gamma_{\psi} I)(\star_{\psi}u_2,\star_{\psi}u_2) e^{\psi} dV-4 \norm{\eta(\star_{\psi}u_2)}_{-\psi}^2.
\end{multline*}
Since $\psi(z)=t|z|^2$, we have \[(H\psi-\gamma_{\psi} I)(\star_{\psi}u_2,\star_{\psi}u_2)=t((n-q)-(n-1))\abs{\star_\psi u_2}^2,\]
 so
\begin{equation} \label{eq-last1}
  \norm{\dbar(\star_{\psi}u_2)}_{-\psi}^2+ \norm{\dbar^*_{-\psi}(\star_{\psi}u_2)}^2_{-\psi} \geq t(q-1)\norm{\star_{\psi}u_2}^2_{-\psi}-4 \norm{\eta(\star_{\psi} u_2)}_{-\psi}^2.
\end{equation}
A standard calculation  (cf. \cite[Lemma~2]{ChSh12})  shows that we have formally (i.e. as differential operators on smooth functions) that
\[ \dbar^*_{\psi}= -\star_{-\psi}\,\dbar\, \star_\psi
\quad \text{ and }
 \dbar = -\star_{-\psi} \dbar^*_{-\psi}\star_{\psi}. \]
 Therefore using \eqref{eq:hodge_star2} we see that
 \[ \norm{\dbar(\star_{\psi}u_2)}_{-\psi} = \norm{- \star_{-\psi}\dbar(\star_{\psi}u_2)}_{\psi}=\norm{\dbar^*_{\mathrm{mix},\psi}u_2}_\psi\]
 and similarly
 \[ \norm{\dbar^*_{-\psi}(\star_{\psi}u_2)}_{-\psi}= \norm{- \star_{-\psi}\dbar^*_{-\psi}(\star_{\psi}u_2)}_{\psi}  =\norm{\dbar_{\mathrm{mix}} u_2}_{\psi}.\]
Therefore \eqref{eq-last1} is equivalent to
\[
  \norm{\dbar^*_{\mathrm{mix},\psi}  u_2}_{\psi}^2+ \norm{\dbar_{\mathrm{mix}} u_2}^2_{\psi}\geq t(q-1)\norm{u_2}^2_{\psi}-4 \norm{\eta u_2}_{\psi}^2.
\]
Define the constant $B$ associated to the annulus $\Omega$ by
\begin{equation}\label{eq-Bdef}
B = \sup_{U\setminus \Omega_2} 4 \eta^2,
\end{equation}
which is finite since by assumption  the open set $U$ was so chosen that  $\rho \in \mathcal{C}^2(\ol{U}\setminus \Omega_2)$.
This gives us
\begin{equation}
\label{eq:key_estimate_2}
  \norm{\dbar^*_{\mathrm{mix},\psi}  u_2}_{\psi}^2+ \norm{\dbar_{\mathrm{mix}} u_2}^2_{\psi}\geq (t(q-1)-B)\norm{u_2}^2_{\psi}.
\end{equation}
Since $\sqrt{a^2+b^2}\leq |a|+|b|$, we may assume that $t>\frac{B}{q-1}$ and rewrite \eqref{eq:key_estimate_1} and \eqref{eq:key_estimate_2} in the form
\begin{align*}
  \norm{\dbar_{\mathrm{mix}} u_1}_{{\psi}}+\norm{\dbar^*_{\mathrm{mix},{\psi}} u_1}_{{\psi}}&\geq
  \sqrt{tq}\norm{u_1}_{{\psi}},\\
  \norm{\dbar_{\mathrm{mix}} u_2}_{{\psi}}+\norm{\dbar^*_{\mathrm{mix},{\psi}} u_2}_{{\psi}}&\geq
  \sqrt{t(q-1)-B}\norm{u_2}_{{\psi}}.
\end{align*}
Note that
\begin{equation}
\label{eq:cutoff_estimate}
  \norm{\dbar_{\mathrm{mix}} u_2}_{{\psi}}\leq\norm{\dbar_{\mathrm{mix}} u}_{{\psi}}+\sup_{\Omega}|\overline\nabla\chi|\norm{u}_{{\psi}},
\end{equation}
with similar estimates for $\norm{\dbar_{\mathrm{mix}} u_1}_{{\psi}}$, $\norm{\dbar^*_{\mathrm{mix},{\psi}} u_2}_{{\psi}}$, and $\norm{\dbar^*_{\mathrm{mix},{\psi}} u_1}_{{\psi}}$, so since $|\overline\nabla\chi|=\frac{1}{2}|\nabla\chi|$, we have
\begin{align*}
  \norm{\dbar_{\mathrm{mix}} u}_{{\psi}}+\norm{\dbar^*_{\mathrm{mix},{\psi}} u}_{{\psi}}+\sup_{\Omega}|\nabla\chi|\norm{u}_{{\psi}}&\geq
  \sqrt{tq}\norm{u_1}_{{\psi}},\\
  \norm{\dbar_{\mathrm{mix}} u}_{{\psi}}+\norm{\dbar^*_{\mathrm{mix},{\psi}} u}_{{\psi}}+\sup_{\Omega}|\nabla\chi|\norm{u}_{{\psi}}&\geq
  \sqrt{t(q-1)-B}\norm{u_2}_{{\psi}}.
\end{align*}
Since $tq>t(q-1)-B$, we use the same constant in the lower bound of each estimate.  Consequently, if we  let
\begin{equation}\label{eq-Adef}
A=\sup_{\Omega}\abs{\nabla\chi}
\end{equation} we have
\[
  \norm{u}_{{\psi}}\leq\left(\frac{2}{\sqrt{t(q-1)-B}}\right)\left(\norm{\dbar_{\mathrm{mix}} u}_{{\psi}}+\norm{\dbar^*_{\mathrm{mix},{\psi}} u}_{{\psi}}+A\norm{u}_{{\psi}}\right).
\]
If $t>\frac{4A^2+B}{q-1}$, then we can rearrange terms and obtain
\begin{equation}
\label{eq:closed_range_mixed}
  \norm{u}_{{\psi}}\leq\frac{2}{\sqrt{t(q-1)-B}-2A}\left(\norm{\dbar_{\mathrm{mix}} u}_{{\psi}}+\norm{\dbar^*_{\mathrm{mix},{{\psi}}} u}_{{\psi}}\right).
\end{equation}
By   \cite[Lemma 3.1]{LiSh13}, the estimate \eqref{eq:closed_range_mixed} will apply to any $u\in\dom(\dbar_{\mathrm{mix}})\cap\dom(\dbar^*_{\mathrm{mix},{\psi}})$.

It now follows that for $2\leq q \leq n$, the operator $\dbar_{\rm mix}$ satisfies the analog of the ``basic estimate'' of $L^2$-theory (i.e.,  \cite[equation (4.3.16)]{ChSh01}).
Then the result follows using a classical argument (cf. \cite[Theorem 4.3.4]{ChSh01}). We include the details for completeness.


For $2\leq q\leq n$, suppose $f\in L^2_{p,q}(\Omega)\cap\dom(\dbar_{\mathrm{mix}})$ such that $\dbar_{\mathrm{mix}}f=0$.  For an arbitrary $g\in L^2_{p,q}(\Omega)\cap\dom\dbar^*_{\mathrm{mix},{\psi}}$, we may decompose $g=g_1+g_2$, where $g_1\in(\ker\dbar_{\mathrm{mix}})^\bot$,  $g_2\in\ker\dbar_{\mathrm{mix}}$, and  the orthogonal complement is taken with respect to the weighted inner product.  Then
\[
  \abs{\left(f,g\right)_{{\psi}}}=\abs{\left(f,g_2\right)_{{\psi}}}\leq\norm{f}_{{\psi}}\norm{g_2}_{{\psi}}.
\]
Applying \eqref{eq:closed_range_mixed} to $\norm{g_2}_{{\psi}}$ and using the fact that $g_1\in\overline{\range\dbar^*_{\mathrm{mix},{\psi}}}\subset\ker\dbar^*_{\mathrm{mix},{\psi}}$, we have
\[
  \norm{g_2}_{{\psi}}\leq\frac{2}{\sqrt{t(q-1)-B}-2A}\norm{\dbar^*_{\mathrm{mix},{{\psi}}} g_2}_{{\psi}}=\frac{2}{\sqrt{t(q-1)-B}-2A}\norm{\dbar^*_{\mathrm{mix},{{\psi}}} g}_{{\psi}}.
\]
Hence,
\[
  \abs{\left(f,g\right)_{{\psi}}}\leq\frac{2}{\sqrt{t(q-1)-B}-2A}\norm{f}_{{\psi}}\norm{\dbar^*_{\mathrm{mix},{{\psi}}} g}_{{\psi}}.
\]
As a result, there is a bounded linear functional $\dbar^*_{\mathrm{mix},{{\psi}}} g\mapsto\left(f,g\right)_{{\psi}}$ on $\range\dbar^*_{\mathrm{mix},{\psi}}$ with norm bounded by $\frac{2}{\sqrt{t(q-1)-B}-2A}\norm{f}_{{\psi}}$.  Using the Hahn-Banach Theorem to extend this to the closure of $\range\dbar^*_{\mathrm{mix},{\psi}}$ and then applying the Riesz Representation Theorem, there exists an element $u\in L^2_{p,q-1}(\Omega)$ lying in the closure of $\range\dbar^*_{\mathrm{mix},{\psi}}$ such that for every $g\in L^2_{p,q}(\Omega)\cap\dom\dbar^*_{\mathrm{mix},{\psi}}$ we have
\[
  (f,g)_{{\psi}}=(u,\dbar^*_{\mathrm{mix},{\psi}}g)_{{\psi}}
\]
and
\begin{equation} \label{eq-prediam}
  \norm{u}_{{\psi}}\leq\frac{2}{\sqrt{t(q-1)-B}-2A}\norm{f}_{{\psi}}.
\end{equation}
It immediately follows that $\dbar_{\mathrm{mix}}u=f$.  After a translation, we may assume that $|z|\leq\delta$ on $\overline\Omega$, where
\begin{equation}\label{eq-diam} \delta = \sup_{z,z'\in \Omega} \abs{z-z'}\end{equation}
is the diameter of the set $\Omega$.  Hence, \eqref{eq-prediam} gives us
\[
  \norm{u}\leq\frac{2e^{t\delta^2}}{\sqrt{t(q-1)-B}-2A}\norm{f}.
\]
We may check that the resulting constant is optimized when $t=\frac{\left(A+\sqrt{A^2+\frac{q-1}{2\delta^2}}\right)^2+B}{q-1}$.  Hence, we have
\begin{equation}
\label{eq:closed_range_mixed_dbar}
  \norm{u}\leq C_q(\Omega,\mathrm{mix})\norm{f},
\end{equation}
where
\begin{equation}
\label{eq:mixed_dbar_constant}
  C_q(\Omega,\mathrm{mix})\leq\frac{2}{\sqrt{A^2+\frac{q-1}{2\delta^2}}-A}\exp\left(\frac{\delta^2}{q-1}\left(\left(A+\sqrt{A^2+\frac{q-1}{2\delta^2}}\right)^2+B\right)\right).
\end{equation}
In \eqref{eq:mixed_dbar_constant}, $A$ and $B$ are given by \eqref{eq-Adef} and \eqref{eq-Bdef} respectively, and $\delta$ is given by \eqref{eq-diam}.

We now have proved the special case of Proposition~\ref{prop:dbar_mix_solvability} when the boundary is $\mathcal{C}^2$.

\subsection{\texorpdfstring{$W^1$}{W1} Estimates}
\label{sec:W1_estimates}

Before we can prove Proposition \ref{prop:dbar_mix_solvability} when the boundary is $\mathcal{C}^{1,1}$, we will need to prove $W^1$ estimates for the hole.  This is because we do not necessarily have an exhaustion for the annulus, but we are guaranteed to have an exhaustion for the hole.  In order to facilitate the study of the hole using results on the annulus, we will need a standard extension result for Lipschitz domains.  In $\mathbb{R}^n$, a bounded Lipschitz domain $D$ is defined to be a domain with the property that $bD$ can be covered by a finite open cover $\{U_j\}_{j\in J}$ with the property that on each $U_j$ there exist orthonormal coordinates $\{x_1,\ldots,x_n\}$ and a Lipschitz function $\psi_j:\mathbb{R}^{n-1}\rightarrow\mathbb{R}$ with Lipschitz constant $M_j$ such that
\[
  D\cap U_j=\{x\in U_j:x_n>\psi_j(x_1,\ldots,x_{n-1})\}.
\]
Let $\{\chi_j\}_{j\in J}$ denote a partition of unity subordinate to $\{U_j\}_{j\in J}$.  The irregularity of a Lipschitz domain can be estimated by the quantities
\begin{equation}
\label{eq:Lipschitz_constants}
  \sup_{j\in J}M_j\text{ and }\sup_{j\in J}|\nabla\chi_j|.
\end{equation}
\begin{lem}
\label{lem:extension_operator}
  Let $\Omega_1,\Omega_2\subset\mathbb{C}^n$ such that $\overline\Omega_2\subset\Omega_1$, and suppose that $\Omega_2$ has Lipschitz boundary.    Then there exists an operator $E:W^1_{p,q}(\Omega_2)\rightarrow W^1_{p,q}(\Omega_1)$ such that
  \begin{enumerate}
    \item $Ef|_{\Omega_2}=f$ for any $f\in W^1_{p,q}(\Omega_2)$,
    \item $Ef$ is compactly supported in $\Omega_1$, and
    \item \begin{equation}\label{eq:extension_estimate}\norm{Ef}_{W^1}\leq E(\Omega_1,\Omega_2)\norm{f}_{W^1}\end{equation}
    for any $f\in W^1_{p,q}(\Omega_2)$, where $E(\Omega_1,\Omega_2)>0$ is a constant depending only on the two {quantities} in  \eqref{eq:Lipschitz_constants} and on  $\dist(b\Omega_1,b\Omega_2)$.
  \end{enumerate}
\end{lem}
The proof of this lemma can be found in   \cite[Section VI.3]{Ste70}, which also proves that $E$ can be constructed to be continuous in $W^s_{p,q}$ for all $s\geq 0$ (although the resulting constant also depends on the dimension $n$), or  in \cite[Section~5.4]{Eva10}, which provides a fairly explicit construction in the $s=1$ case that we will need.  In particular, although the proof in \cite{Eva10} is only stated for $C^1$ domains, we note that the boundary of a Lipschitz domain can be locally straightened using a Lipschitz map, which will have uniformly bounded derivatives almost everywhere, and this suffices for estimates in $W^1$.

\begin{proof}[Proof of Theorem \ref{thm:W^1_solvability}:]

  We first assume that the boundary of $D$ is $\mathcal{C}^2$ and weakly $q$-convex.  Let $\Omega_1$ be a smooth, bounded, strictly pseudoconvex domain such that $\overline{D}\subset\Omega_1$ and let $\Omega=\Omega_1{\setminus} \overline D$.  Let $E:W^1_{p,q}(D)\rightarrow W^1_{p,q}(\Omega_1)$ be the operator given by Lemma \ref{lem:extension_operator}.  Given $f$ as in the hypotheses of the Theorem, we set $\tilde f=E f$.  Since $W^1_{p,q}(\Omega_1)\subset\dom(\dbar)$, we may set $g=\dbar\tilde f$.  Note that $g\in L^2_{p,q+1}(\Omega_1)$ and $g|_{D}=0$, so $g|_{\Omega}\in\dom(\dbar_{\mathrm{mix}})$.  When $2\leq q+1\leq n$, since we are assuming that $\Omega$ has a $\mathcal{C}^2$ boundary we may apply Theorem \ref{prop:dbar_mix_solvability} to $g$ and obtain $v\in L^2_{p,q}(\Omega)\cap\dom(\dbar_{\mathrm{mix}})$ satisfying $\dbar_{\mathrm{mix}}v=g$ and
  \begin{equation}
  \label{eq:v_estimate}
    \norm{v}_{L^2(\Omega)}\leq C_q(\Omega,\mathrm{mix})\norm{g}_{L^2(\Omega)}=C_q(\Omega,\mathrm{mix})\norm{\dbar\tilde f}_{L^2(\Omega)}\leq C_q(\Omega,\mathrm{mix})\norm{\tilde f}_{W^1(\Omega)}.
  \end{equation}
  When $q=n$ (and hence $g\equiv 0$), we let $v\equiv 0$.  We extend $v$ to $\tilde v\in\Omega_1$ by setting $\tilde v$ equal to zero on $\overline D$.  Since $\tilde v\in\dom(\dbar_{\mathrm{mix}})$, it follows that $\tilde v\in\dom\dbar\cap L^2_{p,q}(\Omega_1)$.  Set $h=\tilde f-\tilde v$.  Then $h\in\dom(\dbar)\cap L^2_{p,q}(\Omega_1)$ and
  \[
    \dbar h=\dbar\tilde f-\dbar\tilde v=g-g=0
  \]
  on $\Omega_1$.  Let $\tilde u\in L^2_{p,q-1}(\Omega_1)$ denote the canonical (i.e., $L^2$-minimal) solution to $\dbar\tilde u=h$ on $\Omega_1$.  Let $u=\tilde u|_{D}$.  On $D$, we have
  \[
    \dbar u=h|_{D}=\tilde f|_{D}-\tilde v|_{D}=f-0=f.
  \]
  Furthermore, the interior regularity of the $\dbar$-Neumann operator gives us
  \begin{equation}\label{eq-intelliptic}
    \norm{u}_{W^1(D)}\leq K(\Omega_1,D)\norm{h}_{L^2(\Omega_1)}
  \end{equation}
  for some constant $K(\Omega_1,D)>0$.  We note that this is an actual estimate, not an a priori estimate, and so we have $u\in W^1_{p,q-1}(D)$.  Continuing to estimate this norm, we have
  \begin{align*}
    \norm{u}_{W^1(D)}&\leq K(\Omega_1,D)\norm{\tilde f-\tilde v}_{L^2(\Omega_1)}\\
    &\leq K(\Omega_1,D)\left(\norm{\tilde f}_{L^2(\Omega_1)}+\norm{\tilde v}_{L^2(\Omega_1)}\right)\\
    &\leq K(\Omega_1,D)\left(\norm{\tilde f}_{W^1(\Omega_1)}+\norm{v}_{L^2(\Omega)}\right).
  \end{align*}
  Using \eqref{eq:v_estimate}, we have
  \begin{align*}
    \norm{u}_{W^1(D)}&\leq K(\Omega_1,D)\left(\norm{\tilde f}_{W^1(\Omega_1)}+C(\Omega,\mathrm{mix})\norm{\tilde f}_{W^1(\Omega)}\right)\\
    &\leq K(\Omega_1,D)\left(1+C_q(\Omega,\mathrm{mix})\right)E(\Omega_1,D)\norm{f}_{W^1(D)}.
  \end{align*}
  Hence, \eqref{eq:W^1_estimate} follows with the constant
  \begin{equation}
  \label{eq:W^1_constant_estimate}
    C\leq K(\Omega_1,D)\left(1+C(\Omega,\mathrm{mix})\right)E(\Omega_1,D).
  \end{equation}

  Now, suppose the boundary of $D$ is merely $\mathcal{C}^{1,1}$.  By regularizing the $q$-subharmonic function $-\log(-\rho)+C|z|^2$ given by Definition \ref{defn:weak_q_convex}, we may exhaust $D$ by smooth, weakly $q$-convex domains on which the $\mathcal{C}^2$-norm of the defining function for each signed distance function is uniformly bounded  (see \cite{fassina} for details).  Each constant in \eqref{eq:W^1_constant_estimate} is uniformly bounded with respect to this norm; in particular, the parameter $B$ in \eqref{eq:mixed_dbar_constant} is uniformly bounded with respect to the $\mathcal{C}^2$ norm, since $\eta$ depends only on two derivatives of the defining function.  Hence, we obtain \eqref{eq:W^1_estimate} on each domain in the exhaustion with uniform bounds on the associated constant.  Using a standard argument with weak limits, we obtain a solution to $\dbar u=f$ on $D$ satisfying \eqref{eq:W^1_estimate}.

\end{proof}

\subsection{Proof of Proposition~\ref{prop:dbar_mix_solvability}}

We imitate the proof of \cite[Theorem~2.2]{LiSh13}.  Briefly, we extend $f$ to $\tilde f\in L^2_{p,q}(\Omega_1)$ by setting $\tilde f=0$ on $\overline\Omega_2$. Since $f\in\dom(\dbar_{\mathrm{mix}})$, $\dbar\tilde f=0$ on $\Omega_1$.  Let $v\in L^2_{p,q-1}(\Omega_1)$ be the canonical  solution to $\dbar v=\tilde f$ on $\Omega_1$, as in Lemma~\ref{lem-qconvex}.  By interior elliptic regularity, $v\in W^1_{p,q-1}(\overline\Omega_2)$.  Since $\dbar v=0$ on $\Omega_2$, we may use Theorem \ref{thm:W^1_solvability} on each connected component of $\Omega_2$ to find $w\in W^1_{p,q-2}(\Omega_2)$ such that $\dbar w=v$ on $\Omega_2$.  We extend $w$ to $\tilde w\in W^1_{p,q-2}(\Omega_1)$, and let $u=v-\dbar\tilde w$. Then $u\in L^2_{p,q-1}(\Omega_1)$, $\dbar u=f$ on $\Omega$, and $u=0$ on $\Omega_2$, so $u\in\dom(\dbar_{\mathrm{mix}})$.

\subsection{Non-closed range for \texorpdfstring{$q=1$}{q=1} in annuli with mixed boundary conditions}

We note that the range of $q$ in Proposition~\ref{prop:dbar_mix_solvability} is sharp, in the sense that closed range for $\dbar_{\mathrm{mix}}$ fails when $q=1$.  In fact, we have the following:
\begin{prop}\label{prop-nonclosed}
Let $\Omega$ be an annulus  in $\cx^n$ with $n\geq 1$ in which the envelope is a bounded
pseudoconvex domain, and the hole is Lipschitz.
Then for $0\leq p \leq n$ the densely defined operator
\[ \dbar_{\rm mix}:L^2_{p,0}(\Omega)\rightarrow L^2_{p,1}(\Omega)\]
does not have closed range.
\end{prop}
\begin{rem}
If the hole $\Omega_2$ is also pseudoconvex, one can further show that the range of the operator $\dbar_{\rm mix}$ is dense in the space $\ker (\dbar)\cap L^2_{0,1}(\Omega_2)$ of $\dbar$-closed forms, using an Oka-Weil type approximation theorem.
\end{rem}

\begin{proof}We use the idea of the proof of   \cite[Theorem~2.4]{LiSh13}. Consider the cohomology  vector space defined by the $\dbar_{\rm mix}$ operator:
\[ H^{p,1}_{\rm mix}(\Omega) =
\frac{\ker (\dbar_{\rm mix})\cap L^2_{p,1}(\Omega)}{ \mathrm{range}(\dbar_{\rm mix})\cap L^2_{p,1}(\Omega)}, \]
which can be given a natural quotient topology. This quotient topology is Hausdorff if and only if the range of
$\dbar_{\rm mix}$ in $L^2_{p,1}(\Omega)$ is closed.  Therefore, to prove the result it suffices to show that $ H^{p,1}_{\rm mix}(\Omega)$ is not Hausdorff.

As usual, let us denote the envelope by $\Omega_1$ and the hole by $\Omega_2$. Let
\[\mathcal{O}L^2_p(\Omega_1) = L^2_{p,0}(\Omega_1)\cap \ker \dbar\]
be the space of square integrable holomorphic $p$-forms on  the envelope $\Omega_1$. For $p=0$ this is just the Bergman space of $\Omega_1$.

 Let $\mathcal{O}W^1_p(\Omega_2)$ be the space
\begin{equation}\label{eq-holosobolev}
 \mathcal{O}W^1_p(\Omega_2)= W^1_{p,0}(\Omega_2)\cap \ker \dbar,\end{equation}
consisting of holomorphic $p$-forms on the hole $\Omega_2$
with coefficients which belong to the Sobolev space
$W^1(\Omega_2)$.  Notice that both $\mathcal{O}L^2_p(\Omega_1)$ and $ \mathcal{O}W^1_p(\Omega_2)$ are Hilbert spaces.
We consider the following sequence of  topological vector spaces and continuous linear maps
\begin{equation}\label{eq-exact}0 \rightarrow \mathcal{O}L^2_p(\Omega_1)\xrightarrow{R} \mathcal{O}W^1_p(\Omega_2)\xrightarrow{\ell} H^{p,1}_{\rm mix}(\Omega)\rightarrow 0,\end{equation}
where $R$ is the restriction map  $R:\mathcal{O}L^2_p(\Omega_1)\to \mathcal{O}W^1_p(\Omega_2)$ given by
$f\mapsto f|_{\Omega_2}$, and the map $\ell:  \mathcal{O}W^1_p(\Omega_2)\to H^{p,1}_{\rm mix}(\Omega) $  is defined in the following way. Since $\Omega_2$ has Lipschitz boundary, Lemma \ref{lem:extension_operator} will apply. Define $\ell$ by setting
\[ \ell(f) = \left[(\dbar Ef)|_{\Omega}\right].\]
Then $\ell:\mathcal{O}W^1_p(\Omega_2)\to H^{p,1}_{\rm mix}(\Omega) $ is a continuous linear map between topological vector spaces, where the quotient topology on $ H^{p,1}_{\rm mix}(\Omega) $  need not be Hausdorff.  We claim that the map $\ell$ is defined independently of the extension
operator $E$. Indeed, if $\wt{f}$ is any extension of $f$ as a function in $W^1(\Omega_1)$, and $\wt{\gamma}= \left[(\dbar \wt{f})|_{\Omega}\right]\in H^{p,1}_{\rm mix}(\Omega)$, then $\ell(f)- \wt{\gamma} = \left[(\dbar (Ef- \wt{f}))|_{\Omega}\right]= \left[\dbar\left( (Ef-\wt{f})|_{\Omega}\right)\right]=0$, since $(Ef-\wt{f})|_{\Omega}\in \dom(\dbar_{\rm mix})$.

We now claim that \eqref{eq-exact} is an exact sequence of topological vector spaces and
continuous linear map, and that the image of $R$ in $\mathcal{O}W^1_p(\Omega_2)$ is not closed. Then we have a linear homeomorphism of topological vector spaces (see \cite{cassa71})
\[ H^{p,1}_{\rm mix}(\Omega) \cong\frac{\mathcal{O}W^1_p(\Omega_2)}{ R(\mathcal{O}L^2_p(\Omega_1))}.\]
Since the image of $R$ is not closed, the quotient on the right hand side is not Hausdorff, and
therefore $H^{p,1}_{\rm mix}(\Omega)$ is also not Hausdorff.

Since $\Omega_1$ is connected, it follows by the identity principle that $R$ is injective, and
therefore the sequence is exact at $\mathcal{O}L^2_p(\Omega_1)$. We now show that  $R:\mathcal{O}L^2_p(\Omega_1)\to \mathcal{O}W^1_p(\Omega_2)$ does not have closed range. To see this first note that $R$ is compact: by the Bergman inequality,  a sequence of functions uniformly bounded in $\mathcal{O}L^2_p(\Omega_1)$  has a uniform sup norm bound on an open set $D$ such that $ \Omega_2 \Subset D \subset \Omega_1$. Therefore, by Montel's theorem, there is a subsequence converging in the $\mathcal{C}^1$ topology on $\ol{\Omega_2}$ which means it converges also in $W^1$. Now if
$R$ has closed range, then by the open mapping theorem, $R$ is an open map from $\mathcal{O}L^2_p(\Omega_1)$ onto $\mathcal{O}L^2_p(\Omega_1)|_{\Omega_2}$. From the compactness of $R$, the unit ball of $R(\mathcal{O}L^2_p(\Omega_1))\subset \mathcal{O}W^1_p(\Omega_2)$ with respect to the $W^1$ norm is relatively compact, which implies that $R(\mathcal{O}L^2_p(\Omega_1))$ is finite dimensional. But since the holomorphic polynomials belong to $R(\mathcal{O}L^2_p(\Omega_1))$, this is absurd, and $R$ does not have closed range.

We now claim that
$ \ker (\ell)= R(\mathcal{O}L^2_p(\Omega_1)),$
which shows exactness at $\mathcal{O}W^1_p(\Omega_2)$. If $f\in \ker(\ell)$, there is an extension
$\wt{f}\in W^1(\Omega_1)$. By hypothesis, on $\Omega$ there is a $u\in \dom(\dbar_{\rm mix})\cap L^2(\Omega)$ such that  $\dbar \wt{f}=\dbar u$ on $\Omega$. Let $\wt{u}$ be the extension by 0 of $u$ to $\Omega_1$, i.e $\wt{u}=u$ on $\Omega$ and $\wt{u}=0$ in
$\ol{\Omega_1}$. Then the function $F=\wt{f}-\wt{u}$ is a holomorphic extension of $f$ to
$\Omega_1$ and $F\in \mathcal{O}L^2_p(\Omega_1)$ so that $R(F)=f$. On the other hand, if $f\in R(\mathcal{O}L^2_p(\Omega_1))$, there is a $F\in \mathcal{O}L^2_p(\Omega_1)$ so that $R(F)=f$. Then $\ell(f) =[ \dbar F|_{\Omega}]=0.$  The claim follows.


To complete the proof,  we now claim that $\ell$ is surjective, so that the sequence is exact at $H^{p,1}_{\rm mix}(\Omega)$. Indeed, given a class $\gamma\in  H^{p,1}_{\rm mix}(\Omega)$ represented by $g\in \dom(\dbar_{\rm mix})\cap L^2_{0,1}(\Omega)$ with $\dbar g=0$, we can extend $g$ by 0 on $\ol{\Omega_2}$ to obtain a $\wt{g}\in L^2_{(p,1)}(\Omega_1)$ such that $\dbar \wt{g}=0$. Let $v$ be the canonical solution of $\dbar v=g$,
and let $f=v|_{\Omega_2}$. Then clearly $f\in \mathcal{O}W^1_p(\Omega_2)$, and $\ell(f)= \left[(\dbar Ef)|_{\Omega}\right]= \left[(\dbar v)|_{\Omega}\right]= \gamma.$

\end{proof}

\section{\texorpdfstring{$L^2$}{L2}-estimates on annuli}
In this section we give proofs of Theorems~\ref{thm:q_n-q}, \ref{thm:q-convex} and \ref{thm-infinitedim}. One of our main tools will be duality arguments in the $L^2$ setting (cf. \cite{ChSh12}). For an annulus $\Omega$, as usual we let $\Omega_1$ denote its envelope and
$\Omega_2$ denote its hole.
\subsection{Closed Range: Proof of Theorem~\ref{thm:q_n-q}}
 We will  need the following easy fact (cf. \cite[Lemma~3]{ChSh12}):
\begin{lem}\label{lem:cs12lem3}
On a domain $D$,  for some $p,q$, the operator
\[ \dbar: L^2_{p,q-1}(D)\to L^2_{p,q}(D)\]
has closed range if and only if the strong minimal realization
\[ \dbar_c: L^2_{n-p,n-q}(D)\to L^2_{n-p,n-q+1}(D)\]
has closed range. Further,  if these operators have closed range, the constants in both closed
range estimates are the same. More precisely: if $C$ is a constant   such that  for each $u\in \dom(\dbar)\cap (\ker \dbar)^\perp\cap L^2_{p,q-1}(D)$ we have
\begin{equation}\label{eq:closedrangeestdbar}
\norm{u}_{L^2(D)} \leq C \norm{\dbar u}_{L^2(D)},
\end{equation}
 which is equivalent to closed range of $\dbar$  by  \cite[Theorem 1.1.1]{Hor65},
 then we also have for each $v \in\dom\dbar_c\cap(\ker\dbar_c)^\bot\cap L^2_{n-p,n-q}(\Omega) $ that
\begin{equation}\label{eq:closedrangeestdbarc}
\norm{v}_{L^2(D)} \leq C \norm{\dbar_c v}_{L^2(D)},
\end{equation}
where the constants are the same in \eqref{eq:closedrangeestdbar} and \eqref{eq:closedrangeestdbarc}
\end{lem}
\begin{proof} This is essentially \cite[Lemma~3]{ChSh12}, which follows from the representation  $\dbar^*= -* \cdot\dbar_c\cdot * $, along with  \cite[Theorem 1.1.1]{Hor65},
which states that the range an operator is closed if and only of the range of its adjoint is closed, and the best constants in the estimates corresponding to the closed range (i.e. \eqref{eq:closedrangeestdbar}  and \eqref{eq:closedrangeestdbarc} here) are the same.
\end{proof}
\begin{lem}\label{lem-qconvex2} Let $D\Subset \cx^n$ be a bounded weakly  $q$-convex domain, where $1\leq q \leq n$ and let $\delta =\sup_{z,z'\in D}\abs{z-z'}$ be the diameter of $D$. Let $0\leq p \leq n$.  Then the minimal realization
\[ \dbar_c: L^2_{p, n-q}(D)\to L^2_{p, n-q+1}(D)\]
has closed range, and consequently the cohomology with minimal realization  $H^{p, n-q+1}_{c, L^2}(D)$ is Hausdorff. Further, the constant in the closed range estimate is the same as in \eqref{eq-poincare}, i.e., whenever $f\in {\rm range}(\dbar_c)\cap  L^2_{p, n-q+1}(D)$,
 there is a solution $v\in \dom(\dbar_c)\cap L^2_{p, n-q}(D) $ of $\dbar_cv =f$ such that
 \begin{equation}\label{eq-v1}
 \norm{v}_{L^2(D)} \leq \sqrt{\frac{e}{q}}\cdot \delta \cdot \norm{f}_{L^2(D)}.
 \end{equation}
 Further, the coefficients of the solution $v$ lie locally in the Sobolev space $W^1$, and for each relatively compact open subset $D_2 \Subset D$ there is
 a constant $K(D,D_2)$ such that
 \begin{equation}\label{eq-v2}
 \norm{v}_{W^1(D_2)} \leq K(D, D_2) \cdot \norm{u}_{L^2(D)}.
 \end{equation}
\end{lem}
\begin{proof} This is a generalization of \cite[Theorem~9.1.2]{ChSh01} and the proof is the same. In fact, we can take $v= -\star N_{n-p, q}\dbar \star$, where $N_{n-p, q}$ is the $\dbar$-Neumann operator in degree $(n-p,q)$ which is easily seen to exist
 on the $q$-convex domain $D$. Then \eqref{eq-v1} follows, and the constant follows from Lemma~\ref{lem:cs12lem3} above. Finally, \eqref{eq-v2} follows from the interior regularity of
 the $\dbar$-Neumann operator (see \cite{ChSh01}).
 \end{proof}
\begin{proof}[Proof of Theorem~\ref{thm:q_n-q}]
By   Lemma~\ref{lem:cs12lem3}, it suffices to prove that for every $u\in L^2_{n-p,n-q}(\Omega)\cap\dom\dbar_c^\Omega\cap(\ker\dbar_c^\Omega)^\bot$, we have
\begin{equation}
\label{eq:dbar*_closed_range_annulus}
  \norm{u}_{L^2(\Omega)}\leq C_q(\Omega,L^2)\norm{\dbar_c^\Omega u}_{L^2(\Omega)}.
\end{equation}
Let $u\in L^2_{n-p,n-q}(\Omega)\cap\dom\dbar_c^\Omega\cap(\ker\dbar_c^\Omega)^\bot$, and  define $f\in L^2_{n-p,n-q+1}(\Omega_1)$ by setting $f=\dbar_c^\Omega u$ on $\Omega$ and $f=0$ on $\overline\Omega_2$.  By Lemma \ref{lem:extension_str_min}, $f\in\range\dbar_c^{\Omega_1}$, since we may also extend $u$ to be zero on $\overline\Omega_2$. Therefore, by Lemma~\ref{lem-qconvex2} above there is {\em another} solution $v\in L^2_{n-p,n-q}(\Omega_1)\cap\dom(\dbar_c^{\Omega_1})$  of the equation  $\dbar_c^{\Omega_1}v=f$ such that we have properties \eqref{eq-v1} and \eqref{eq-v2}, i.e.,
\[
    \norm{v}_{L^2(\Omega_1)}\leq \sqrt{\frac{e\delta^2}{q}}\norm{f}_{L^2(\Omega_1)},
  \]
  and
  \[
      \norm{v}_{W^1(\Omega_2)}\leq K(\Omega_1,\Omega_2)\norm{f}_{L^2(\Omega_1)}.
  \]
On $\Omega_2$, we have $\dbar v=0$.  We know that $n-1\geq n-q\geq 1$, so we can apply Theorem \ref{thm:W^1_solvability} to each connected component of $\Omega_2$ and see that there exists $g\in W^1_{n-p,n-q-1}(\Omega_2)$ such that $\dbar g=v$ on $\Omega_2$ and
  \[
    \norm{g}_{W^1(\Omega_2)}\leq C\norm{v}_{W^1(\Omega_2)}.
  \]
  Define $\tilde g\in W^1_{n-p,n-q-1}(\Omega_1)$ by setting $\tilde g=E g$, where $E$ is given by Lemma \ref{lem:extension_operator}.  Hence
  \[
    \norm{\tilde g}_{W^1(\Omega_1)}\leq E(\Omega_1,\Omega_2)\norm{g}_{W^1(\Omega_2)}.
  \]
  Then $\tilde g\in\dom\dbar_c^{\Omega_1}$ and $\dbar_c^{\Omega_1}\tilde g|_{\Omega_2}=v$.  Hence, $v-\dbar_c^{\Omega_1}\tilde g\in\dom\dbar_c^{\Omega}$ and $\dbar_c^{\Omega}(v-\dbar_c^{\Omega_1}\tilde g)=f|_{\Omega}$.  Since $u$ is the $L^2$ minimal solution to $\dbar_c^\Omega u=f|_{\Omega}$, we must have
  \begin{align*}
    \norm{u}_{L^2(\Omega)}&\leq\norm{v-\dbar_c^{\Omega_1}\tilde g}_{L^2(\Omega)}\\
    &\leq\norm{v}_{L^2(\Omega)}+\norm{\tilde g}_{W^1(\Omega)}\\
    &\leq\norm{v}_{L^2(\Omega)}+E(\Omega_1,\Omega_2)\norm{g}_{W^1(\Omega_2)}\\
    &\leq\norm{v}_{L^2(\Omega)}+E(\Omega_1,\Omega_2)C_{n-q-1}(\Omega_2,W^1)\norm{\tilde v}_{W^1(\Omega_2)}\\
    &=\norm{v}_{L^2(\Omega)}+E(\Omega_1,\Omega_2)C_{n-q-1}(\Omega_2,W^1)\norm{v}_{W^1(\Omega_2)}\\
    &\leq\left(\sqrt{\frac{e}{q}}\cdot \delta+E(\Omega_1,\Omega_2)C_{n-q-1}(\Omega_2,W^1)K(\Omega_1,\Omega_2)\right)\norm{\dbar_c^\Omega u}_{L^2(\Omega)},
  \end{align*}
  from which \eqref{eq:dbar*_closed_range_annulus} follows, where
  \begin{equation}
C_q(\Omega, L^2)= \sqrt{\frac{e}{q}}\cdot \delta+E(\Omega_1,\Omega_2)C_{n-q-1}(\Omega_2,W^1)K(\Omega_1,\Omega_2).
\label{eq-cql2}
\end{equation}
\end{proof}
  \subsection{Vanishing of Cohomology: Proof of Theorem~\ref{thm:q-convex}}

We begin with an adaptation of Lemma \ref{lem-qconvex2} to the degree in which we obtain a solution operator for $\dbar_c$:
\begin{lem}\label{lem-qconvex3} Let $D\Subset \cx^n$ be a bounded weakly  $q$-convex domain, where $1\leq q \leq n$ and let $\delta =\sup_{z,z'\in D}\abs{z-z'}$ be the diameter of $D$. Let $0\leq p \leq n$.  Then whenever $f\in L^2_{n-p, n-q}(D)\cap\dom\dbar_c^D$ satisfies $\dbar_c^D f=0$,
 there is a solution $v\in \dom(\dbar_c^D)\cap L^2_{n-p, n-q-1}(D) $ of $\dbar_c^D v =f$ such that
 \begin{equation}\label{eq-v3}
 \norm{v}_{L^2(D)} \leq \sqrt{\frac{e}{q}}\cdot \delta \cdot \norm{f}_{L^2(D)}.
 \end{equation}
 Further, the coefficients of the solution $v$ lie locally in the Sobolev space $W^1$, and for each relatively compact open subset $D_2 \Subset D$ there is
 a constant $K(D,D_2)$ such that
 \begin{equation}\label{eq-v4}
 \norm{v}_{W^1(D_2)} \leq K(D, D_2) \cdot \norm{u}_{L^2(D)}.
 \end{equation}
\end{lem}
\begin{proof} This is a generalization of \cite[Theorem~9.1.2]{ChSh01} and the proof is the same. In fact, we can take $v= -\star\dbar N_{p, q} \star$, where $N_{p, q}$ is the $\dbar$-Neumann operator in degree $(n-p,q)$ which is easily seen to exist
 on the $q$-convex domain $D$. Then \eqref{eq-v1} follows, and the constant follows from Lemma~\ref{lem:cs12lem3} above. Finally, \eqref{eq-v2} follows from the interior regularity of
 the $\dbar$-Neumann operator (see \cite{ChSh01}).
 \end{proof}

We next prove a solvability result for $\dbar_c$ on the annulus.  We will state this using the language of cohomology, so we define
\[ H^{p,q}_{c,L^2}(D)= \frac{\ker \dbar_c^D: L^2_{p,q}(D)\to L^2_{p,q+1}(D)}{{\rm range}\, \dbar_c^D: L^2_{p,q-1}(D)\to L^2_{p,q}(D) }\]

\begin{prop}\label{prop:q-convex}  Let   $\Omega$ be an annulus  in $\mathbb{C}^n$, $n\geq 2$,  such that for some $0\leq q\leq n-2$ the envelope is weakly $(q+1)$-convex and the
hole is weakly $(n-q-1)$-convex. Then $H^{n-p,n-q}_{c,L^2}(\Omega)=0$.
\end{prop}

\begin{proof}
Let $f\in L^2_{n-p,n-q}(\Omega)\cap\dom\dbar_c^\Omega$ satisfy $\dbar_c^\Omega f=0$.  Since $f\in\dom\dbar_c^\Omega$, we may extend $f$ to $\tilde f\in L^2_{n-p,n-q}(\Omega)\cap\dom\dbar_c^{\Omega_1}$ by setting $\tilde f=f$ on $\Omega$ and $\tilde f=0$ on $\overline\Omega_2$.  By Lemma \ref{lem:extension_str_min}, $f\in\dom\dbar_c^{\Omega_1}$.  We may use Lemma~\ref{lem-qconvex3} to find $v\in L^2_{n-p,n-q-1}(\Omega_1)\cap\dom(\dbar_c^{\Omega_1})$  solving the equation  $\dbar_c^{\Omega_1}v=\tilde f$ such that we have properties \eqref{eq-v3} and \eqref{eq-v4}, i.e.,
\[
    \norm{v}_{L^2(\Omega_1)}\leq \sqrt{\frac{e\delta^2}{q}}\norm{f}_{L^2(\Omega_1)},
  \]
  and
  \[
      \norm{v}_{W^1(\Omega_2)}\leq K(\Omega_1,\Omega_2)\norm{f}_{L^2(\Omega_1)}.
  \]
On $\Omega_2$, we have $\dbar v=0$.  We know that $n-1\geq n-q-1\geq 1$, so we can apply Theorem \ref{thm:W^1_solvability} to each connected component of $\Omega_2$ and see that there exists $g\in W^1_{n-p,n-q-2}(\Omega_2)$ such that $\dbar g=v$ on $\Omega_2$ and
  \[
    \norm{g}_{W^1(\Omega_2)}\leq C\norm{v}_{W^1(\Omega_2)}.
  \]
  Define $\tilde g\in W^1_{n-p,n-q-2}(\Omega_1)$ by setting $\tilde g=Eg$, where $E$ is given by Lemma \ref{lem:extension_operator}.  Hence
  \[
    \norm{\tilde g}_{W^1(\Omega_1)}\leq E(\Omega_1,\Omega_2)\norm{g}_{W^1(\Omega_2)}.
  \]
  Then $\tilde g\in\dom\dbar_c^{\Omega_1}$ and $\dbar_c^{\Omega_1}\tilde g|_{\Omega_2}=v$.  Hence, $v-\dbar_c^{\Omega_1}\tilde g\in\dom\dbar_c^{\Omega}$ and $\dbar_c^{\Omega}(v-\dbar_c^{\Omega_1}\tilde g)=f$.
\end{proof}

To conclude the proof of Theorem \ref{thm:q-convex}, we will need the $L^2$-Serre duality from \cite{ChSh12}:
\begin{thm}[Chakrabarti-Shaw]
\label{thm:L2_Serre_Duality}
  Let $D\subset\mathbb{C}^n$ be a domain.  The following are equivalent:
  \begin{enumerate}
    \item The operators $\dbar:L^2_{p,q-1}(D)\rightarrow L^2_{p,q}(D)$ and $\dbar:L^2_{p,q}(D)\rightarrow L^2_{p,q+1}(D)$ have closed range.

    \item The map $\star:L^2_{p,q}(D)\rightarrow L^2_{n-p,n-q}(D)$ induces a conjugate-linear isomorphism of $H^{p,q}_{L^2}(D)$ with $H^{n-p,n-q}_{c,L^2}(D)$.
  \end{enumerate}
\end{thm}

\begin{proof}[Proof of Theorem \ref{thm:q-convex}:]
  According to Theorem \ref{thm:L2_Serre_Duality}, we can prove $H^{p,q}_{L^2}(\Omega)=0$ if we can show $\dbar$ has closed range in $L^2_{p,q}(\Omega)$ and $L^2_{p,q+1}(\Omega)$, and $H^{n-p,n-q}_{c,L^2}(\Omega)=0$.

  We note that weak $q$-convexity percolates up, so that weakly $q$-convex domains are also weakly $(q+1)$-convex, and weakly $(n-q-1)$-convex domains are also weakly $(n-q)$ convex.  Therefore, Theorem \ref{thm:q_n-q} guarantees that $\dbar$ has closed range in $L^2_{p,q}(\Omega)$ and $L^2_{p,q+1}(\Omega)$.  Since Proposition \ref{prop:q-convex} guarantees that $H^{n-p,n-q}_{c,L^2}(\Omega)=0$, we are done.
\end{proof}

\subsection{Infinite dimensional \texorpdfstring{$\dbar$}{dbar}-Cohomology: Proof of Theorem~\ref{thm-infinitedim}} We begin with the following dual statement, which is of interest in its own right:

\begin{prop}\label{prop-hcinfdim}
 For $n\geq2$, let  $\Omega_1\subset \cx^n$ be a bounded weakly $(n-1)$-convex domain and $\Omega_2$ a relatively compact  open subset of $\Omega_1$ with Lipschitz boundary.  Let $\Omega = \Omega_1 \setminus \ol{\Omega_2}$. Then there is a natural isomorphism  (i.e., a linear homeomorphism of topological vector spaces)
of the  space \[   \mathcal{O}W^1_{p}(\Omega_2)= W^1_{p,0}(\Omega_2)\cap \{\ker \dbar\}\]
of holomorphic $p$-forms on $\Omega_2$ with coefficients in the Sobolev space $W^1(\Omega_2)$  with the cohomology space of the annulus with minimal realization
\[
  H^{p,1}_{c,L^2}(\Omega):=(L^2_{p,1}(\Omega)\cap\ker\dbar_c^{\Omega})/(\range\dbar_c^\Omega).
\]
Consequently, the space $ H^{p,1}_{c,L^2}(\Omega)$ is Hausdorff and infinite dimensional.
\end{prop}
\begin{proof}
Define a map
\[ \ell: \mathcal{O}W^1_{p}(\Omega_2)\to H^{p,1}_{c, L^2}(\Omega)\]
in the following way.  Let $E$ be given by Lemma \ref{lem:extension_operator}, and define, for $f\in \mathcal{O}W^1_{p}(\Omega_2)$
\[\ell(f) = \left[ (\dbar E f)|_\Omega\right] \in H^{p,1}_{c, L^2}(\Omega),\]
which is easily checked to be defined independently of the choice of the extension operator: in fact we could have taken any extension of $f$ to a form in $W^1_{p,0}(\Omega_1)$ to define $\ell$.

We claim that this map $\ell$ is a linear homeomorphism of topological vector spaces where
$\mathcal{O}W^1_{p}(\Omega_2)$ is given its natural norm topology as a closed subspace of
the space $W^1_{(p,0)}(\Omega_2)$ of $(p,0)$-forms on $\Omega_2$ with coefficients in  the Sobolev space $W^1(\Omega_2)$ on hole, and the cohomology space with minimal realization
 $H^{p,1}_{c, L^2}(\Omega)$ on the annulus is given its natural quotient topology. Notice that at this point we do not know whether $H^{p,1}_{c, L^2}(\Omega)$ is Hausdorff or not.

We see from the definition that $\ell$ is a continuous linear map between (possibly non-Hausdorff) topological vector spaces. We will show that $\ell$ is a set theoretic bijection, and its
inverse map is continuous, which will prove our claim.

To see that $\ell$ is injective, let $f\in \ker \ell$ so that there is a $u\in \dom(\dbar_c^\Omega)\cap L^2_{p,0}(\Omega)$  such that $\dbar_c^\Omega u =\dbar E f$. Let $g=Ef-u$ on $\Omega$, which is  therefore in $L^2_{p,0}(\Omega)$. Since $\dbar g=0$, it follows that $g$ is a holomorphic $p$-form on $\Omega$, and therefore, by Hartogs phenomenon, extends to a holomorphic $p$-form
$\wt{g}$ on the envelope $\Omega_1$ of the annulus. On $\Omega$, we can therefore write
\[ u=Ef -\wt{g}|_\Omega.\]
Now by construction, $Ef$ has compact support in $\Omega_1$, which means that in a neighborhood of the outer boundary $b\Omega_1$ of the annulus, we have $u=- \wt{g}$. But
since $u\in \dom(\dbar_c^\Omega)$, it follows that $\wt{g}\in \dom(\dbar_c^{\Omega_1}) \cap L^2_{p,0}(\Omega_1)$. Therefore if we extend $\wt{g}$ by 0 outside $\Omega_1$ by zero, then for the extended form we have $\dbar \wt{g}=0$ on $\cx^n$, which means that the extended form is a
holomorphic form  on $\cx^n$ with compact support. So $\wt{g}\equiv 0$, so that $u=Ef$ on
the annulus $\Omega$ and the extended function $Ef|_\Omega\in \dom(\dbar_c^\Omega)$.  Let $\wt{f}$ be the function on $\Omega_1$ obtained by extending $Ef|_\Omega$ by zero  on $\ol{\Omega_2}$.  Then $\wt{f}\in \dom(\dbar_c^{\Omega_1})$.
Therefore $Ef-\wt{f}$ is also in  $\dom(\dbar_c^{\Omega_1})$. Now on $\Omega_2$
this function $Ef-\wt{f}$ coincides with $f$ and on $\Omega$ it vanishes, and it is holomorphic
since $\dbar(Ef-\wt{f})=0$. It follows now that $f\equiv 0$, so the map $\ell$ is injective.

To show that $\ell$ is a surjective, let $[g]\in H^{p,1}_{c,L^2}(\Omega)$ be a cohomology class,
where $g\in \dom(\dbar_c^\Omega)\cap L^2_{p,1}(\Omega)$ with $\dbar g=0$. Extending by 0 on $\Omega_2$ , we obtain a form $\wt{g}$ on $\Omega_1$ which clearly belongs to $\dom(\dbar_c^{\Omega_1})\cap L^2_{p,1}(\Omega_1)$, and in fact, $\dbar_c^{\Omega_1} \wt{g}=0$. By Lemma~\ref{lem-qconvex}, since $\Omega_1$ is $(n-1)$-convex, there is a $u\in \dom(\dbar_c^{\Omega_1})\cap L^2_{p,0}(\Omega_1)$ such
that $\dbar u =\wt{g}$. Let $f=u|_{\Omega_2}$. Then clearly $f\in \mathcal{O}W^1_{p}(\Omega_2) $, and $\ell(f)=[\dbar u]= [g]$.

Therefore, the map $\ell$ is a continuous bijection, and $\ell^{-1}$ is given by   $\ell^{-1}([g])=
u|_{\Omega_2}$, which is easily seen to be a continuous map. Therefore $\ell$ is a linear homeomorphism of topological vector spaces, in fact an invertible linear map between Hilbert spaces.
  \end{proof}

  \begin{proof}[Proof of Theorem~\ref{thm-infinitedim}] Thanks to Theorem~\ref{thm:q-convex} with $q=n-1$, since the envelope is assumed to be $(n-1)$-convex and the hole to be pseudoconvex,
  the $\dbar$-operator
  \[ \dbar: L^2_{p, n-2}(\Omega)\to L^2_{p,n-1}(\Omega),\]
  has closed range.  Also, as in any domain, $\dbar: L^2_{p, n-1}(\Omega)\to L^2_{p,n}(\Omega)$ has closed range. Then we can apply Theorem~\ref{thm:L2_Serre_Duality} to conclude that the map
  \[ \star: H^{p, n-1}_{L^2}(\Omega)\to H^{n-p,1}_{c,L^2}(\Omega),\]
  is a conjugate linear isomorphism. By Proposition~\ref{prop-hcinfdim} above,
  the target is infinite dimensional and Hausdorff, and therefore the same is true for $H^{p, n-1}_{L^2}(\Omega)$, as needed.

  \end{proof}

  Using the definitions of the Hodge-star operator   and that of the map $\ell$ of Proposition~\ref{prop-hcinfdim},  Theorem~\ref{thm-infinitedim} can be restated in the following more precise form:
  \begin{cor}\label{cor-pairing}  Let $\Omega=\Omega_1\setminus\ol{\Omega_2}$ be an annulus in $\cx^n, n\geq 3$, where the envelope $\Omega_1$ is weakly $(n-1)$-convex and the hole $\Omega_2$ has $\mathcal{C}^{1,1}$-boundary and is pseudoconvex.  The bilinear pairing
  \[ H^{p,n-1}_{L^2}(\Omega)\times \mathcal{O}W^1_{n-p}(\Omega_2)\to \cx\]
  given  for $g\in L^2_{p,n-1}(\Omega)\cap \ker( \dbar)$ and $f\in\mathcal{O}W^1_{n-p}(\Omega_2)$ by
  \begin{equation}\label{eq-newpairing}
   [g], f \mapsto \int_\Omega g\wedge (\dbar E f)
   \end{equation}
  identifies the dual of $H^{p,n-1}_{L^2}(\Omega)$ with $ \mathcal{O}W^1_{n-p}(\Omega_2)$, where $E$ is the extension operator in Lemma~\ref{lem:extension_operator}.

  \end{cor}
  The fact that there is a duality between the cohomology $H^{p,n-1}_{L^2}(\Omega)$ of the
  annulus and the space $\mathcal{O}W^1_{n-p}(\Omega_2)$ of holomorphic $(n-p)$-forms
  on the hole with Sobolev $W^1$ coefficients was first observed in the case when $\Omega_2$ has $\mathcal{C}^\infty$-smooth boundary, and $\Omega_1$ is pseudoconvex by Shaw \cite[Theorem~3.3]{Sha10}.
 The pairing \eqref{eq-newpairing} in Corollary~\ref{cor-pairing} is identical to  the boundary pairing used in \cite{Sha10}, since
 \[ \int_\Omega g \wedge \dbar Ef = \int_\Omega d(g\wedge Ef) = -\int_{b\Omega_2}  g\wedge f,\]
 where we have used the fact that $Ef$ has compact support in $\Omega_1$, and the negative
 sign results from the orientation of $b\Omega_2$ as the boundary of $\Omega_2$, rather than a part of the boundary of $\Omega$.

\bibliographystyle{amsplain}
\bibliography{harrington}

\end{document}